\documentclass[11pt]{amsart}
\usepackage{amsfonts}
\usepackage{amsmath}
\usepackage{amssymb}
\usepackage{multicol}
\usepackage{stmaryrd}
\usepackage{cite}
\usepackage{epsfig}
\usepackage{color}
\usepackage{graphics}
\usepackage{graphicx}
\usepackage{epstopdf}
\usepackage{multicol,graphics}
\newcommand\bes{\begin{eqnarray}}
\newcommand\ees{\end{eqnarray}}

\newtheorem{theorem}{Theorem}[section]
\newtheorem{lemma}[theorem]{Lemma}
\newtheorem{corollary}[theorem]{Corollary}

\newtheorem{remark}[theorem]{Remark}
\newtheorem{proposition}[theorem]{Proposition}
\numberwithin{equation}{section}
\allowdisplaybreaks

\begin{document}
\title[Spatial Propagation in a Nonlocal Epidemic Model]{\textbf{Spatial Propagation in an Epidemic Model with  Nonlocal Diffusion: the Influences of Initial Data and Dispersals}}

\author[W.-B. Xu, W.-T. Li and S. Ruan]{Wen-Bing Xu$^{1,2}$, Wan-Tong Li$^{2,*}$ and Shigui Ruan$^{3}$}
\thanks{\hspace{-.6cm}
$^1$Academy of Mathematics and Systems Science, Chinese Academy of Sciences, Beijing 100190, P. R. China.
\\
$^2$School of Mathematics and Statistics, Lanzhou University, Lanzhou, Gansu 730000, P. R. China.
\\
$^3$Department of Mathematics, University of Miami, Coral Gables, FL 33146, USA.
\\
$^*${\sf Corresponding author} (wtli@lzu.edu.cn)}

\date{\today}

\begin{abstract}
This paper studies  an epidemic model with nonlocal dispersals. We focus on the influences of initial data and nonlocal dispersals on its spatial propagation. Here the initial data stand for  the spatial concentrations  of infectious agent and infectious human population  when the epidemic breaks out and the nonlocal dispersals mean their diffusion strategies.
Two types of initial data  decaying to zero exponentially or faster are considered.
For the first type, we show that the spreading speeds are two constants whose signs change with the number of   elements in some set.
Moreover, we find an interesting phenomenon: the asymmetry of nonlocal dispersals can influence the  propagating directions of solutions and  the stability of steady states. For the second type, we show that the spreading speed is decreasing with respect to  the exponentially decaying rate of   initial data, and further, its minimum value coincides with the spreading speed for the first type. In addition, we give some results about the nonexistence of traveling wave solutions and the  monotone property of solutions. Finally, some applications are presented to illustrate the theoretical results.

\vspace{1em}

\end{abstract}

\maketitle

\section{Introduction}
\noindent

To model the spread of cholera in the European Mediterranean regions in 1973, Capasso and Maddalena \cite{CM1981,CM1982} proposed a system of two parabolic differential equations to describe a positive feedback interaction between the concentration of bacteria and the infectious human population; namely, the high concentration of bacteria leads to the large infection rate of human population and once infected the human population increases the growth rate of bacteria. Capasso and Wilson \cite{CW1997,CK1988} also applied this mechanism to model   other epidemics with fecal-oral  transmission (such as typhoid fever  and hepatitis A).  In these studies, the spatial movements of the infectious agent and the infectious human host are described by the Laplacian operators.

In this paper, we use nonlocal convolution operators to represent the spatial movements of the infectious agent and the infectious human host. Then the epidemic model becomes
\begin{equation}\label{1.1}
\left\{\begin{aligned}
& u_t(t,x)=\mathcal{D}_1 u(t,x)-\alpha u(t,x)+h(v(t,x)), ~~t>0,~x\in\mathbb R,\\
& v_t(t,x)=\mathcal{D}_2 v(t,x)-\beta v(t,x)+g(u(t,x)), ~~~t>0,~x\in\mathbb R,\\
& u(0,x)=u_0(x),~v(0,x)=v_0(x),~~x\in\mathbb R,
\end{aligned}
\right.
\end{equation}
where $u(t,x)$ and $v(t,x)$ biologically stand for the spatial concentration of infectious agent (bacteria or viruses)  and the spatial density of infectious human population at time $t$ and  location $x\in \mathbb R$, respectively. The constants $\alpha>0$ and $\beta>0$ denote the natural death rates of infectious agent and infectious humans, respectively. The function $h(v)$ denotes the growth  of infectious agent caused by infectious humans. Meanwhile, the function $g(u)$ is the infection rate of human population under the assumption that the total susceptible human population is a constant  during the evolution of the epidemic. The nonlocal dispersals, represented by the following convolution operators
\[
\begin{aligned}
&\mathcal{D}_1u(t,x)\triangleq k_1*u(t,x)-u(t,x)=\int_{\mathbb R}k_1(x-y)u(t,y)dy-u(t,x),\\
&\mathcal{D}_2v(t,x)\triangleq k_2*v(t,x)-v(t,x)=\int_{\mathbb R}k_2(x-y)v(t,y)dy-v(t,x),
\end{aligned}
\]
describe the  movements of   infectious agent and infectious humans, respectively, between not only  adjacent   but also nonadjacent spatial locations. The dispersal kernel $k_i$   with $i\in\{1,2\}$ is nonnegative and stands for the probability of the movement from the spatial location $0$ to $x$, thus $\int_{\mathbb R}k_i(x)dx=1$.
Here the movements between nonadjacent spatial locations can be thought as the long-distance movements of infectious agent and infectious humans across cites or countries by air-traffic and other long-distance transportation.


\subsection{A brief review of related literature.}
\noindent

The spatial propagation of system \eqref{1.1} and its variants has been widely studied in the literature. For example,  Li et al. \cite{LXZ2017} and Meng et al. \cite{MYH2019} studied the traveling wave solutions, spreading speeds and entire solutions of system \eqref{1.1}. We refer to   Bao and Li  \cite{BL2020}, Bao et al. \cite{BLSW2018}, Hu et al. \cite{HKLL2015}, Liu and Wang \cite{LW2020}, Wang and Castillo-Chavez \cite{WC2012} and Xu et al. \cite{XLL2018} for results on the spreading dynamics of more general nonlocal dispersal systems.
Particularly, if the infected humans do not move during the infectious period (for example, they are in sickbeds or quarantined probably), then system \eqref{1.1} reduces to  the following partially degenerate system
\begin{equation}\label{01.2}
\left\{\begin{aligned}
& u_t(t,x)=k_1*u(t,x)-u(t,x)-\alpha u(t,x)+h(v(t,x)), ~~~t>0,~x\in\mathbb R,\\
& v_t(t,x)=-\beta v(t,x)+g(u(t,x)), ~~~~t>0,~x\in\mathbb R.
\end{aligned}
\right.
\end{equation}
This system is a special case of system \eqref{1.1} with  $k_2(x)$ being equal to a Dirac function $\delta(x)$ (the movement happens only between  every spatial location and itself; namely, there is no movement of  infected humans).
Traveling wave solutions and entire solutions of system \eqref{01.2} were studied by Wang et al. \cite{WLS2018},  Wu and Hsu \cite{WH2016} and Zhang et al. \cite{ZLW2016}. For other related results on nonlocal dispersal epidemic models, we refer to for example Li and Yang \cite{ly2014} and Yang et al.  \cite{YLR2019}.

In addition, if the movements of   infectious agent and infectious human population happen only between adjacent spatial locations,  the classical Laplace diffusion operators are applied  instead of nonlocal dispersal operators.
For the results about classical diffusion epidemic models, we refer to Allen et al. \cite{ABL2007}, Cui et al. \cite{CLL2017}, Cui and Lou \cite{CL2016}, Hsu and Yang \cite{HY2013}, Wang \cite{Wang2011}, Xu and Zhao \cite{XZ2005} and Zhao and Wang \cite{ZW2004}.

Other fundamental properties  involved in this paper such as   existence and uniqueness of solution  in system \eqref{1.1} can be studied following the theories in Andreu-Vaillo et al. \cite{A2010}. The stability of  steady state  can be studied following the techniques in Yang and Li \cite{YL2017}, Yang et al. \cite{YLR2019}, and Zhao and Ruan \cite{ZR2018}.  For more classical results about nonlocal dispersal problems, we refer to Andreu-Vaillo et al. \cite{A2010}, Bates \cite{bates}, Fife \cite{Fife2003}, Kao et al. \cite{KLSH2010}, Li et al. \cite{LSW2010}, Murray \cite{Mur1993},  Shen and Zhang \cite{Shen-2010-JDE}, Wang \cite{Wang2002} and references cited therein.

\subsection{Preview of the main results}
\noindent

In this paper, we mainly study the influences of two important factors on  the spatial propagation in   model  \eqref{1.1}, namely nonlocal dispersals and  initial data.
Here the initial data stand for  the spatial density  of infectious agent and infectious human population  when epidemic breaks out and the nonlocal dispersals mean their diffusion strategies. Our contribution can be summarized in the following three aspects.

First, we consider the dependence of  spatial propagation on the nonlocal dispersals.
Usually, we can find the phenomenon of anisotropic  dispersal; for example,  the avian influenza viruses carried by migratory birds have a higher possibility to move along the flight route.
Then we can use the asymmetric dispersal to study this phenomenon.
Here the asymmetric  dispersal (kernel) means that for any spatial locations $x\in\mathbb R$,  the probability of  organism moving from  $0$ to $x$ is not equal to that  from  $0$ to  $-x$.
Since  diffusion  is the original driving force of spatial propagation, it is necessary to study the changes of spatial propagation caused by the asymmetry of dispersals in system \eqref{1.1}.

Before  it, we recall the known results on the spreading speeds of the following scalar  equation
\begin{equation}\label{01.3}
u_t=k*u-u+f(u),
\end{equation}
where $f(\cdot)$ is  Fisher-KPP type and $k(\cdot)$ is asymmetric.  Then there are  two constants $c_l^*$ and $c_r^*$ such that
\[
\lim\limits_{t\rightarrow+\infty}u(t,x+ct)=1~\text{for}~c_l^*<c<c_r^*,~~~
\lim\limits_{t\rightarrow+\infty}u(t,x+ct)=0~\text{for}~c<c_l^*~\text{or}~c>c_r^*,
\]
where $c_l^*$ and $c_r^*$ are called the {\it spreading speed to left} and {\it right}, respectively, see Lutscher et al. \cite{LPL2005}, Finkelshtein et al. \cite{FKT2015}, Shen and Zhang \cite{Shen-2010-JDE}. Furthermore, Coville et al. \cite{CDM2008} showed that asymmetric kernels  may cause nonpositive minimal wave speed  for traveling wave solutions (see also  Sun et al. \cite{SZLW2019} and Zhang et al. \cite{ZLW2017,ZLWS2019}). As is well known, the minimal wave speed for traveling wave solutions always equals the spreading speed in Fisher-KPP equations. Therefore, it is worth identifying the signs of  spreading speeds when the kernels are asymmetric. Recently, this problem was solved in our paper \cite{XLR2018}, and further, it was shown that the asymmetry level  of kernel  determines the signs of spreading speeds $c_l^*$ and $c_r^*$, which in turn determine the propagating directions of  solutions and influence the stability   of equilibrium states.
 \cite{MHO2019}

Motivated by  \cite{XLR2018,HKLL2015}, we   study the influences of asymmetric kernels on spatial propagation and identify the signs of spreading speeds. However, such a problem is more difficult than that in equation \eqref{01.3}, because the signs of spreading speeds  $c_l^*$ and $c_r^*$  in system \eqref{1.1} are  actually  influenced by two  kernels $k_1(\cdot)$ and $k_2(\cdot)$. In order to treat this problem, we  define
\[
\Lambda=\big\{\lambda\in\mathbb R~\big|~A(\lambda)B(\lambda)\geqslant g'(0)h'(0),~A(\lambda)<0,~B(\lambda)<0\big\},
\]
where
\[
A(\lambda)=\int_\mathbb R k_1(x)e^{\lambda x}dx-1-\alpha,~~B(\lambda)=\int_\mathbb R k_2(x)e^{\lambda x}dx-1-\beta.
\]
Then we show that the signs of $c_l^*$ and $c_r^*$ change with the number of elements in the set $\Lambda$ (see Theorem \ref{th2.2}) which is essentially determined  by the dispersal kernels $k_1(\cdot)$ and $k_2(\cdot)$. Particularly, when $k_1(\cdot)$ and $k_2(\cdot)$ are symmetric,  it follows that  $c^*\triangleq c_r^*=-c_l^*>0$.

We show that in system \eqref{1.1}, the asymmetric dispersals can influence the propagating directions of solutions and  the stability of steady states. More precisely,   denote the spatial region
\begin{equation}\label{01.4}
\Omega(t)\triangleq\{x\in\mathbb R~|~(u(t,x),v(t,x))\geqslant (\nu,\nu)\}~~\text{for}~t\geqslant0~~\text{with some}~\nu\in(0,1),
\end{equation}
and there is an interesting phenomenon:   $\Omega(t)$ propagates to both the left and the right of the $x$-axis for $c_l^*<0<c_r^*$; propagates only to the right for $0<c_l^*<c_r^*$; and propagates only to the left for $c_l^*<c_r^*<0$. For some appropriate initial data, when $c_l^*<0<c_r^*$, the steady state $(u,v)\equiv(1,1)$ is stable; namely  $(u(t,x),v(t,x))\to (1,1)$ as $t\to +\infty$, but when $0<c_l^*<c_r^*$ or $c_l^*<c_r^*<0$, we see that $(u(t,x),v(t,x))\to (0,0)$ as $t\to +\infty$  in any bounded spatial region.

Next, we study the dependence of  spatial propagation on  initial data.
Consider two types of initial data  which decay to zero exponentially or faster as $|x|\rightarrow+\infty$, but their decaying rates are different.
We establish a relationship between the spreading speed and the exponentially decaying rate $\lambda$ of initial data.
For the first type  whose decaying rate is large (this type includes compactly supported functions),  we show that the spreading speeds are   constants $c_l^*$ and $c_r^*$ (see Theorem \ref{th2.3}). For the second type  whose decaying rate is small, when $k_1(\cdot)$ and $k_2(\cdot)$ are symmetric, we show that the spreading speed $c(\lambda)$ is decreasing with respect to  $\lambda$, and  the minimum value of $c(\lambda)$ coincides with $c^*$ (see Theorem \ref{th2.5}).
In addition, we  obtain two other results of system \eqref{1.1}, namely the nonexistence of traveling wave solutions (Corollary \ref{th2.4}) and the  monotone property  of solutions (Theorem \ref{th2.6}).

These results give us guidance for  better control of the spatial propagation of epidemics. We see that even though the spatial concentration of the infectious agent and the infectious human population are very low at the spatial locations far away from $x=0$, they have an important influence on the spatial propagation of system \eqref{1.1}. Therefore, in order to slow down the spreading speed of epidemics,  the prevention  in low-density spatial regions is at least as important as   the treatment in high-density spatial regions.
In addition,  there are some applications of the theoretical results  to the control of  epidemics whose infectious agent is carried by migratory birds.
As we shall see in Section 5, it is possible that the epidemic spreads only along the flight route of  migratory birds and the spatial propagation against the flight route fails,
as long as the infectious humans are kept from moving frequently.

Finally, we show that  the spreading speed in this paper is studied by applying the comparison principle (see Lemma \ref{th4.1}) and constructing new   types of upper and lower solutions, instead of the classic  theories of spreading speeds which are established by Weinberger \cite{Wei1982} and developed by Lewis et al.\cite{LLW2002}, Li et al. \cite{LWL2005}, Liang and Zhao \cite{LZ2008,LZ2010}, Lui \cite{Lui1989}, Yi and Zou \cite{YZ2015}. Indeed, when we study the dependence of spreading speeds on   initial data, the upper and lower solutions method is more useful because it can deal with more general types of initial data (see e.g. Hamel and  Nadin \cite{HN2012}, Hamel and Roques \cite{HR2010} and Xu et al. \cite{XLR2018}).
We  present a new method to construct the lower solution of system \eqref{1.1} which spreads at a speed of $c_1$ or $c_2$, where $c_1\in(c_r^*-\epsilon,c_r^*)$ and $c_2\in(c_l^*,c_l^*+\epsilon)$. We also apply the new  {\it ``forward-backward spreading''} method which  was first given in our previous paper \cite{XLR2018}.
In this method, for any time $T>0$ and  any $\mu\in[0,1]$,  we   construct a lower solution $U_1(t,x)$ in the first period of time $[0,\mu T]$ which spreads at a speed of $c_1$, and in the second period of time $[\mu T,T]$ we   construct another lower solution $U_2(t,x)$ which spreads at a speed of $c_2$ and satisfies  $U_2(\mu T,x)\leqslant U_1(\mu T,x)$.
Then these two lower solutions can be regarded as a lower solution  defined in the time period [0,T]   whose speed is $\bar{c}=\mu c_1+(1-\mu)c_2$. Moreover, the arbitrariness of $\mu$ guarantees that $\bar{c}$ can be  any number in $[c_1, c_2]$.

All the  methods in this paper are applicable to the following $m$-species nonlocal dispersal cooperative system
\begin{equation}\label{01.5}
\left\{
\begin{aligned}
&\partial_t U(t,x)=K*U(t,x)-U(t,x)+F(U(t,x)),~~t>0,~x\in\mathbb{R},\\
&U(0,x)=U_0(x)=(u_{0,1}(x),\cdots,u_{0,m}(x)),~x\in\mathbb{R},
\end{aligned}
\right.
\end{equation}
where $m\geqslant2$, $U(t,x)=(u_1(t,x),\cdots,u_m(t,x))$ and  $K(x)=(k_1(x),\cdots,k_m(x))$. Here the function $F(U)=(f_1(U),\cdots,f_m(U))$ is cooperative and $F'(\mathbf 0)$ is an irreducible  matrix. Then it could be shown that the initial data $U_0(\cdot)$ and the dispersal kernel $K(\cdot)$ have   similar influences on the spreading speeds of system \eqref{01.5}. Actually, system \eqref{1.1} can be regarded as a special case of system \eqref{01.5} with $m=2$. The study of system \eqref{1.1} has  simpler calculations, but it shows clearer presentations of the new upper and lower solutions and the ``forward-backward spreading'' method.
Moreover, in system \eqref{01.5}, if the nonlocal dispersal operators are replaced by Laplacian  operators, all the methods still work. However, it is not necessary to apply the ``forward-backward spreading'' method, since we can use a monotone property    similar to Theorem  \ref{th2.6} instead (also see the proof of Theorem \ref{th2.5} for more details).

The rest of this paper is organized as follows. In Section 2, we present the definitions and some mathematical analysis of spreading speeds. Section 3  is devoted to the spatial propagation  for the first type of initial data and asymmetric kernels. In Section 4, we study the spatial propagation  for the second type of initial data and symmetric kernels. Meanwhile, we also prove some  monotone property result for system \eqref{1.1}. In Section 5, we give some applications of the theoretical results.

\section{The signs of spreading speeds}
\noindent

In this section, we define the notations of spreading speeds and  identify their signs.
First, we give some assumptions. Let $\alpha$ and $\beta$ be two positive constants. Throughout this paper,  we assume $g(\cdot)$ and $h(\cdot)$ are  two functions in $C^1([0,1])\cap C^{1+\delta_0}([0,p_0])$, where $\delta_0$ and $p_0$ are two constants in $(0,1)$, and satisfy that
\begin{itemize}
\item[(H1)] $g(0)=h(0)=0$, $h(1)/\alpha=g(1)/\beta=1$, $h\left(g(s)/\beta\right)-\alpha s>0$ for all $s\in(0,1)$;
\item[(H2)] $0<g(u)\leqslant g'(0)u$ and  $g'(u)\geqslant0$ for all $u\in(0,1)$,  $0<h(v)\leqslant h'(0)v$ and $h'(v)\geqslant 0$ for all $v\in(0,1)$.
\end{itemize}
From (H1) and (H2),   system \eqref{1.1} is monostable and  $(u(t,x),v(t,x))\equiv(1,1)$ is the unique nontrivial steady state. Moreover, we have that $\alpha\beta<h'(0)g'(0)$.
Suppose   $k_1(\cdot)$ and $k_2(\cdot)$ are two continuous and nonnegative dispersal kernel functions  satisfying
\begin{itemize}
  \item[(K1)] $\int_{\mathbb R}k_i(x)dx=1$ and $\int_\mathbb{R}k_i(x)e^{\lambda x}dx<+\infty$ for any $\lambda\in\mathbb R$ and  $i\in\{1,2\}$;
  \item[(K2)]  there are $x_i^+\in\mathbb R^+$ and $x_i^-\in\mathbb R^-$ such that $k_i(x_i^\pm)>0$, for each $i\in\{1,2\}$.
\end{itemize}
We assume the initial data $u_0(\cdot)$ and $v_0(\cdot)$ are two continuous functions which satisfy that   $0\leqslant u_0(x)\leqslant1$, $0\leqslant v_0(x)\leqslant1$ for all  $x\in\mathbb R$ and
\[
u_0(x)\rightarrow0,~~v_0(x)\rightarrow0~\text{as}~|x|\rightarrow+\infty.
\]

Now define
\begin{equation}\label{02.1}
c(\lambda)=\frac{1}{\lambda}D(\lambda)~~\text{for}~\lambda\neq0,
\end{equation}
where
\[
D(\lambda)=\frac{1}{2}\left[A(\lambda)+B(\lambda)+\sqrt{(A(\lambda)-B(\lambda))^2+4 g'(0)h'(0)}\right],
\]
\begin{equation}\label{02.6}
A(\lambda)=\int_\mathbb R k_1(x)e^{\lambda x}dx-1-\alpha,~~B(\lambda)=\int_\mathbb R k_2(x)e^{\lambda x}dx-1-\beta.
\end{equation}
It follows that  $D(\lambda)>A(\lambda)$ and $D(\lambda)>B(\lambda)$ for  $\lambda\in\mathbb R$. Particularly, if  $k_1(\cdot)$ and $k_2(\cdot)$ are symmetric,  then $c(\lambda)=-c(-\lambda)$ for   $\lambda\neq0$.

\begin{theorem}\label{th2.1}
There are two unique  constants $\lambda_r^*\in\mathbb R^+$ and $\lambda_l^*\in\mathbb R^-$ such that
\begin{equation}\label{02.2}
c_r^*\triangleq c(\lambda_r^*)=\inf\limits_{\lambda\in\mathbb R^+}\{c(\lambda)\},~~c_l^*\triangleq c(\lambda_l^*)=\sup\limits_{\lambda\in\mathbb R^-}\{c(\lambda)\},
\end{equation}
and $c'(\lambda)<0$ for   $\lambda\in(\lambda_l^*,0)\cup(0,\lambda_r^*)$. Moreover, we have $c_l^*<c_r^*$.
Particularly, if $k_1(\cdot)$ and $k_2(\cdot)$ are symmetric, then  $c^*\triangleq c_r^*=-c_l^*>0$ and $\lambda^*\triangleq\lambda_r^*=-\lambda_l^*$.
\end{theorem}

\begin{proof}
This proof is based on  some mathematical analysis of the functions $c'(\lambda)$ and $c''(\lambda)$.
First we prove that
\begin{equation}\label{03.1}
\lim\limits_{\lambda\rightarrow0^+}c'(\lambda)=-\infty~~\text{and}~~\lim\limits_{\lambda\rightarrow0^-}c'(\lambda)=-\infty.
\end{equation}
By  some simple  calculations, we see  the functions $A(\lambda)$, $B(\lambda)$, $A'(\lambda)$ and $B'(\lambda)$ are  uniformly bounded as $\lambda\rightarrow 0$. Then the functions  $D(\lambda)$ and $D'(\lambda)$ are also uniformly bounded as $\lambda\rightarrow 0$.
Therefore, we easily get \eqref{03.1}   from $c'(\lambda)=\lambda^{-1}D'(\lambda)-\lambda^{-2}D(\lambda)$ and $D(0)>0$.

Now we show that
\begin{equation}\label{03.2}
c'(\lambda)>0~~~\text{for $|\lambda|$ large enough}.
\end{equation}
From the definitions of functions $c(\lambda)$ and $D(\lambda)$, we have
\[
\begin{aligned}
2\lambda^2c'(\lambda)&=2(\lambda D'(\lambda)-D(\lambda))\\
&=(\lambda A'-A)+(\lambda B'-B)+\frac{(A-B)\big[(\lambda A'-A)-(\lambda B'-B)\big]-4 g'(0)h'(0)}{[(A-B)^2+4 g'(0)h'(0)]^{\frac{1}{2}}}.\\
\end{aligned}
\]
Then from $|A-B|<[(A-B)^2+4 g'(0)h'(0)]^{\frac{1}{2}}$, it follows that
\[
\lambda^2c'(\lambda)>\min\big\{\lambda A'(\lambda)-A(\lambda)-\sqrt{g'(0)h'(0)},~\lambda B'(\lambda)-B(\lambda)-\sqrt{g'(0)h'(0)}\big\}.
\]
By some simple calculations, we have
\[
\begin{aligned}
&\lambda A'(\lambda)-A(\lambda)=\int_\mathbb R k_1(x)e^{\lambda x}(\lambda x-1)dx+1+\alpha\rightarrow+\infty~~\text{as}~|\lambda|\rightarrow+\infty, \\
&\lambda B'(\lambda)-B(\lambda)=\int_\mathbb R k_2(x)e^{\lambda x}(\lambda x-1)dx+1+\beta\rightarrow+\infty~~\text{as}~|\lambda|\rightarrow+\infty,
\end{aligned}
\]
which imply that \eqref{03.2} holds.

Next we try to prove that
\begin{equation}\label{03.3}
\lambda c''(\lambda)>0~\text{for}~\lambda\neq0,~\text{provided}~c'(\lambda)=0.
\end{equation}
Indeed, since  $c''(\lambda)=\lambda^{-1}[D''(\lambda)-2c'(\lambda)]$, we just need to prove that
\begin{equation}\label{03.4}
D''(\lambda)>0~~\text{for all}~~\lambda\in\mathbb R.
\end{equation}
From the definitions of functions $A(\lambda)$, $B(\lambda)$ and $D(\lambda)$, it follows that for all $\lambda\in\mathbb R$,  $A''(\lambda)>0$, $B''(\lambda)>0$ and
\[
2D''=A''+B''+\frac{(A-B)(A''-B'')}{[(A-B)^2+4 g'(0)h'(0)]^{\frac{1}{2}}}+\frac{4h'(0)g'(0)(A'-B')^2}{[(A-B)^2+4 g'(0)h'(0)]^{\frac{3}{2}}}.
\]
By combining with $|A-B|<[(A-B)^2+4 g'(0)h'(0)]^{\frac{1}{2}}$, we get
\[
D''(\lambda)\geqslant \min\{A''(\lambda),B''(\lambda)\}>0 ~~\text{for all}~~\lambda\in\mathbb R.
\]
Then we get \eqref{03.3}.

It follows from \eqref{03.3} that there is at most one constant  $\lambda_r^*$ in $\mathbb R^+$ such that $c'(\lambda_r^*)=0$. Meanwhile, \eqref{03.1} and \eqref{03.2} imply the existence of this constant. Similarly, there is a unique constant $\lambda_l^*\in\mathbb R^-$ such that $c'(\lambda_l^*)=0$.
Therefore, we have
\begin{equation}\label{03.5}
c'(\lambda)\left\{
\begin{aligned}
&>0,&&\lambda\in(-\infty,\lambda_l^*)\cup(\lambda_r^*,+\infty),\\
&=0,&&\lambda=\lambda_l^*~\text{or}~\lambda=\lambda_r^*,\\
&<0,&&\lambda\in(\lambda_l^*,0)\cup(0,\lambda_r^*).
\end{aligned}
\right.
\end{equation}
Then we obtain \eqref{02.2} from \eqref{03.5}. Moreover, since $c'(\lambda)=\lambda^{-1}[D'(\lambda)-c(\lambda)]$ and $c'(\lambda_l^*)=c'(\lambda_r^*)=0$, we have $ c_l^*=c(\lambda_l^*)=D'(\lambda_l^*)~~\text{and}~~c_r^*=c(\lambda_r^*)=D'(\lambda_r^*).$ From \eqref{03.4} and $\lambda_l^*<0<\lambda_r^*$, it follows that  $c_l^*<c_r^*$. Particularly, if  $k_1(\cdot)$ and $k_2(\cdot)$ are symmetric, we have $D(\lambda)=D(-\lambda)$ for $\lambda\in\mathbb R$. Then  $c(\lambda)+c(-\lambda)=0$ for $\lambda\neq 0$, which implies  $\lambda_r^*=-\lambda_l^*$ and $c_r^*=-c_l^*>0$.
\end{proof}

In order to  identify the signs of  $c_l^*$ and $c_r^*$, we define a set
\[
\Lambda\triangleq\big\{\lambda\in\mathbb R~\big|~A(\lambda)B(\lambda)\geqslant g'(0)h'(0),~A(\lambda)<0,~B(\lambda)<0\big\}.
\]
Now we give a relationship between the set $\Lambda$ and the signs of  $c_l^*$ and $c_r^*$.
\begin{theorem}\label{th2.2}
We have either $\Lambda\subseteq \mathbb R^+$ or $\Lambda\subseteq \mathbb R^-$. Moreover,
\begin{itemize}
  \item[(i)] if $\Lambda=\varnothing$, then $c_l^*<0<c_r^*${\rm;}
  \item[(ii)] if  $\Lambda\cap \mathbb R^+$ is a singleton set, then $c_l^*<c_r^*=0${\rm;}
  \item[(iii)] if $\Lambda\cap \mathbb R^-$ is a singleton set, then $0= c_l^*<c_r^*${\rm;}
  \item[(iv)] if~~$\mathrm{int}(\Lambda)\cap \mathbb R^+\neq\varnothing$, then $c_l^*<c_r^*<0${\rm;}
  \item[(v)] if~~$\mathrm{int}(\Lambda)\cap \mathbb R^-\neq\varnothing$, then $0<c_l^*<c_r^*$.
\end{itemize}
\end{theorem}

\begin{proof}
First, we prove that either $\Lambda\subseteq \mathbb R^+$ or $\Lambda\subseteq \mathbb R^-$. Since $A(0)B(0)=\alpha\beta<h'(0)g'(0)$, we have  $0\notin\Lambda$. So it is sufficient  to prove that the set $\Lambda$ is a closed interval in $\mathbb R$. For this purpose, we denote
\[
\Lambda^A=\{\lambda\in\mathbb R~\big|~A(\lambda)<0\}~~\text{and}~~\Lambda^B=\{\lambda\in\mathbb R~\big|~B(\lambda)<0\}.
\]
Then we have $\Lambda\subseteq\Lambda^A\cap\Lambda^B$.
Some   calculations show  that $A''(\lambda)>0$ and $B''(\lambda)>0$ for all $\lambda\in\mathbb R$, which imply that the sets $\Lambda^A$ and $\Lambda^B$ are two open intervals in $\mathbb R$.
For any $\lambda\in\Lambda^A\cap\Lambda^B$, if $\big(A(\lambda)B(\lambda)\big)'=A'(\lambda)B(\lambda)+A(\lambda)B'(\lambda)=0$, then we have
\begin{equation}\label{03.6}
\big(A(\lambda)B(\lambda)\big)''=A''(\lambda)B(\lambda)+A(\lambda)B''(\lambda)+2A'(\lambda)B'(\lambda)<0.
\end{equation}
Therefore, the set $\Lambda$ is a closed interval in $\mathbb R$, which means that either $\Lambda\subseteq \mathbb R^+$ or $\Lambda\subseteq \mathbb R^-$.

Now we determine the signs of $c_l^*$ and $c_r^*$. From the definition of the function $D(\lambda)$, we have
\[
D(\lambda)<0\iff A(\lambda)+B(\lambda)<0~\text{and}~A(\lambda)B(\lambda)> g'(0)h'(0)\iff\lambda\in\mathrm{int}(\Lambda).
\]
Similarly, we can get
\[
D(\lambda)=0\iff A(\lambda)+B(\lambda)<0~\text{and}~A(\lambda)B(\lambda)= g'(0)h'(0)\iff\lambda\in\partial\Lambda.
\]
Then it follows that
\[
D(\lambda)>0\iff\lambda\notin \Lambda.
\]
Therefore, if $\Lambda=\varnothing$, then $D(\lambda)>0$ for all $x\in\mathbb R$, which implies that   $c_l^*<0<c_r^*$. If there is some constant $\lambda_0\in\mathbb R^+$ such that $\Lambda\cap\mathbb R^+=\{\lambda_0\}=\partial\Lambda$,   we have $
c(\lambda_0)=0=\inf\limits_{\lambda\in\mathbb R^+}\{c(\lambda)\}=c_r^*>c_l^*.
$
If there is some constant $\lambda_0\in\mathrm{int}(\Lambda)\cap \mathbb R^+$, then it follows that
$
0>c(\lambda_0)\geqslant c_r^*>c_l^*.
$
Similarly, we can get Theorem \ref{th2.2} (iii) and (v).
\end{proof}

\begin{remark}
\rm
From Theorem \ref{th2.2} we can see that the signs of $c_l^*$ and $c_r^*$ change with the number of elements in the set $\Lambda$, which is essentially determined  by the kernels  $k_1(\cdot)$ and $k_2(\cdot)$. Moreover, from (i) we have  $c_l^*<0<c_r^*$ when
\begin{equation}\label{02.3}
(1+\alpha-E(k_1))(1+\beta-E(k_2))<g'(0)h'(0),
\end{equation}
where $E(k)$ can describe  the asymmetry level of $k(\cdot)$ and is defined by
\[
E(k)=\inf\left\{\int_{\mathbb R}k(x)e^{\lambda x}dx~\big|~\lambda\in\mathbb R\right\}.
\]
It is easy to check that $E(k)\in[0,1]$. Particularly, when $k_1(\cdot)$ and $k_2(\cdot)$ are symmetric,  we have   $E(k_1)=E(k_2)=1$, which verifies that \eqref{02.3} is right by $\alpha\beta<h'(0)g'(0)$.
\end{remark}

\section{First type of initial data and  asymmetric kernels case}
\noindent

In this section, we establish the spatial propagation result of system \eqref{1.1} for the first type of initial data and asymmetric  kernels by constructing  new   types of upper and lower solutions   and using the ``forward-backward spreading'' method.
Now we present the main theorem.

\begin{theorem}\label{th2.3}
Assume that {\rm(H1)}, {\rm(H2)}, {\rm(K1)} and {\rm{(K2)}} hold. If $u_0(\cdot)$ and $v_0(\cdot)$ satisfy  that $u_0(x_0)>0$, $v_0(x_0)>0$ for some constant $x_0\in\mathbb R$ and there are two positive constants  $x_1$ and $\Gamma_0$ such that
\[
\max\big\{u_0(x),~v_0(x)\big\}e^{\lambda_l^*x}\leqslant  \Gamma_0~\text{for}~x \leqslant -x_1,~~\max\big\{u_0(x),~v_0(x)\big\}e^{\lambda_r^*x}\leqslant  \Gamma_0~\text{for}~x \geqslant x_1,
\]
then for any small  $\epsilon>0$ there is a constant $\nu\in(0,1)$ such that  the solution of system \eqref{1.1} has the following  properties:
\[
\left\{\begin{aligned}
&\lim\limits_{t\rightarrow+\infty}\sup\limits_{x-x_0\leqslant(c_l^*-\epsilon)t}(u(t, x),v(t,x))=(0,0),\\
&\inf\limits_{(c_l^*+\epsilon)t \leqslant x-x_0\leqslant(c_r^*-\epsilon)t}(u(t, x),v(t,x))\geqslant(\nu, \nu)~~\text{for all}~~t>0,\\
&\lim\limits_{t\rightarrow+\infty}\sup\limits_{x-x_0\geqslant(c_r^*+\epsilon)t}(u(t, x),v(t,x))=(0,0).
\end{aligned}
\right.
\]
\end{theorem}

Before giving its proof, we show some other results derived from Theorem \ref{th2.3}.
We see that the spreading speeds of system \eqref{1.1} for this type of initial values are $c_l^*$ and $c_r^*$ whose signs are determined by $k_1(\cdot)$ and $k_2(\cdot)$ as stated in Section 2. Therefore, the asymmetric dispersals in system \eqref{1.1} can influence the  propagating directions of solutions and   the stability  property of steady states. More precisely, the spatial region  $\Omega(t)$ defined by \eqref{01.4} propagates to both the left and the right of the $x$-axis for $c_l^*<0<c_r^*$; propagates only to the right for $0<c_l^*<c_r^*$; and propagates only to the left for $c_l^*<c_r^*<0$.
However, if the set $\Omega(t)$ is connected at time $t>0$, in  case of $0=c_l^*<c_r^*$, the movement of the left boundary of $\Omega(t)$ is slower than linearity and we cannot identify its propagating direction. Similarly, we cannot  identify the propagating direction of  the right boundary of $\Omega(t)$ in case of $c_l^*<c_r^*=0$ either. On the other hand,  for this type of initial data, when $c_l^*<0<c_r^*$, the steady state $(u(t,x),v(t,x))\equiv(1,1)$ is stable; namely  $(u(t,x),v(t,x))\to (1,1)$ as $t\to +\infty$, but when $c_l^*<c_r^*<0$ or $0<c_l^*<c_r^*$, we see that $(u(t,x),v(t,x))\to (0,0)$ as $t\to +\infty$ in any bounded spatial region.

From Theorem \ref{th2.3} we also obtain the following spatial propagation phenomenon: any small positive perturbation of the steady state $(u(t,x),v(t,x))\equiv (0,0)$ at some spatial location $x_0\in\mathbb R$ and time $t=0$ (namely $(u(0,x_0),v(0,x_0))>(0,0)$ holds) will spread in the  spatial region
\begin{equation}\label{02.4}
\Omega(t,\epsilon,x_0)\triangleq\{x\in\mathbb R~|~(c_l^*+\epsilon)t \leqslant x-x_0\leqslant(c_r^*-\epsilon)t\}~~\text{for any}~t>0~\text{and small}~\epsilon>0,
\end{equation}
which means that $(u(t,x),v(t,x))>(\mu,\mu)$ for $x\in\Omega(t,\epsilon,x_0)$ and some constant $\mu>0$.
From this result, we can get some nonexistence results of traveling wave solutions of the following system
\begin{equation}\label{02.5}
\left\{\begin{aligned}
&u_t(t,x)=k_1*u(t,x)-u(t,x)-\alpha u(t,x)+h(v(t,x)), ~~t\in\mathbb R,~x\in\mathbb R,\\
&v_t(t,x)=k_2*v(t,x)-v(t,x)-\beta v(t,x)+g(u(t,x)), ~~~t\in\mathbb R,~x\in\mathbb R.
\end{aligned}
\right.
\end{equation}

\begin{corollary}\label{th2.4}
Assume that {\rm(H1)}, {\rm(H2)}, {\rm(K1)} and {\rm{(K2)}} hold. Suppose that $(u(t,x),v(t,x))=(\phi(x-ct),\psi(x-ct))$ is a traveling wave solution of system \eqref{02.5} and  satisfies $(\phi,\psi)\not\equiv(0,0)$.  {\rm(i)} If $(\phi(+\infty),\psi(+\infty))=(0,0)$, then $c\geqslant c_r^*${\rm;} {\rm(ii)} if $(\phi(-\infty),\psi(-\infty))=(0,0)$, then $c\leqslant c_l^*$.
\end{corollary}

\begin{proof}
Let  the initial data  $(u_0(x),v_0(x))$ in system \eqref{1.1} satisfy
\[
(u_0(x),v_0(x))\leqslant(\phi(x),\psi(x))~\text{for}~x\in\mathbb R,~~ (u_0(x_0),v_0(x_0))\gg(0,0)~~\text{for some}~x_0\in\mathbb R.
\]
Then  Theorem \ref{th2.3} and the comparison principle (Lemma \ref{th4.1})  show that for any constant $\epsilon>0$ small enough,
\[
(\phi(x-ct),\psi(x-ct))\geqslant(u(t,x),v(t,x))\geqslant(\nu,\nu)~~\text{for}~t>0,~x\in \Omega(t,\epsilon,x_0),
\]
where $(u(t,x),v(t,x))$ is a solution of system \eqref{1.1} and $\Omega(t,\epsilon,x_0)$ is defined by \eqref{02.4}.

In case (i), we suppose  $c<c_r^*$.
Let $\epsilon$ be small enough such that $0<\epsilon<c_r^*-c$.
By taking a constant $c_0\in\mathbb R$ satisfying $\max\{c,c_l^*+\epsilon\}<c_0<c_r^*-\epsilon$, we get that $x_0+c_0t\in \Omega(t,\epsilon,x_0)$ and
\[
\big(\phi(x_0+c_0t-ct),\psi(x_0+c_0t-ct)\big)\geqslant(\nu,\nu)~~\text{for}~t>0.
\]
It is a contradiction to $(\phi(+\infty),\psi(+\infty))=(0,0)$.
Similarly, we can prove case (ii).
\end{proof}

\begin{remark}
\rm
Corollary \ref{th2.4} shows that there exists no traveling wave solution $(u(t,x),v(t,x))=(\phi(x-ct),\psi(x-ct))$ of system \eqref{02.5} satisfying  $(\phi(+\infty),\psi(+\infty))=(0,0)$ and  $c\in(-\infty,c_r^*)$. Meanwhile, system \eqref{02.5} has no traveling wave solution  satisfying  $(\phi(-\infty),\psi(-\infty))=(0,0)$ and  $c\in(c_l^*,+\infty)$ either.
\end{remark}

Now we focus on the proof of Theorem \ref{th2.3} in the following  three subsections.

\subsection{Preliminaries}
\noindent

The basic tools in the proof of Theorem \ref{th2.3} are the upper and lower solutions method and the following comparison principle of system \eqref{1.1} whose proof can be found in \cite{LXZ2017}.

\begin{lemma}\label{th4.1}  {\rm(Comparison Principle)}
Assume that {\rm(H1)}, {\rm(H2)} and {\rm(K1)} hold. For any $\tau >0$, if the continuous functions $(u_1(t,x), v_1(t,x))$ and $(u_2(t,x), v_2(t,x))$ satisfy
\[
\left\{\begin{aligned}
&\partial_t u_1-k_1*u_1+u_1+\alpha u_1-h(v_1)\geqslant\partial_t u_2-k_1*u_2+u_2+\alpha u_2-h(v_2),\\
&\partial_t v_1-k_2*v_1+v_1+\beta v_1-g(u_1)\geqslant\partial_t v_2-k_2*v_2+v_2+\beta v_2-g(u_2),\\
&u_1(0,x)\geqslant u_2(0,x), ~v_1(0,x)\geqslant v_2(0,x)
\end{aligned}
\right.
\]
for $t\in(0, \tau]$, $x\in\mathbb{R}$, then $(u_1(t,x),v_1(t,x))\geqslant (u_2(t,x),v_2(t,x))$ for   $t\in[0, \tau]$ and $x\in\mathbb{R}$.
\end{lemma}

Next we define some notations.  For $c\in\mathbb R$ and $\lambda\in\mathbb R$,   denote
\begin{equation}\label{04.1}
G(c,\lambda)\triangleq c\lambda-A(\lambda)=c\lambda-\int_\mathbb R k_1(x)e^{\lambda x}dx+1+\alpha,
\end{equation}
\begin{equation}\label{04.2}
H(c,\lambda)\triangleq c\lambda-B(\lambda)=c\lambda-\int_\mathbb R k_2(x)e^{\lambda x}dx+1+\beta.
\end{equation}
From \eqref{02.1}, we get that  for $\lambda\neq0$,
\begin{equation}\label{04.3}
G(c(\lambda),\lambda)=D(\lambda)-A(\lambda)>0,~~H(c(\lambda),\lambda)=D(\lambda)-B(\lambda)>0.
\end{equation}
It follows that  for $\lambda\neq0$,
\begin{equation}\label{04.4}
G(c(\lambda),\lambda)H(c(\lambda),\lambda)=(D(\lambda)-A(\lambda))(D(\lambda)-B(\lambda))=g'(0)h'(0).
\end{equation}
Denote the function
\begin{equation}\label{04.5}
b(\lambda)\triangleq\frac{1}{2h'(0)}\Big[-A(\lambda)+B(\lambda)+\sqrt{(A(\lambda)-B(\lambda)^2+4h'(0)g'(0)}\Big]>0~~\text{for}~\lambda\in\mathbb R.
\end{equation}
When $k_1$ and $k_2$ are symmetric, we have that $b(\lambda)=b(-\lambda)$.
Then we get  from \eqref{02.1} that
\begin{equation}\label{04.6}
b(\lambda)=\frac{G(c(\lambda),\lambda)}{h'(0)}=\frac{g'(0)}{H(c(\lambda),\lambda)}~~\text{for}~\lambda\neq0.
\end{equation}

In the construction of   new   lower solutions, we also need to introduce some new notations.
For any   $\eta\in\big(0,\min\{g'(0),h'(0)\}\big)$, we  define a function
\begin{equation}\label{04.7}
c_\eta(\lambda)=\frac{1}{\lambda}D_\eta(\lambda)~~\text{for}~\lambda\neq0,
\end{equation}
where
\[
D_\eta(\lambda)=\frac{1}{2}\left[A(\lambda)+B(\lambda)+\sqrt{(A(\lambda)-B(\lambda))^2+4(g'(0)-\eta) (h'(0)-\eta)}\right].
\]
Similarly to \eqref{04.4}, we have
\begin{equation}\label{04.8}
G(c_\eta(\lambda),\lambda)H(c_\eta(\lambda),\lambda)=(g'(0)-\eta)(h'(0)-\eta)~~~\text{for}~~\lambda\neq0.
\end{equation}
By  the same  method used in the proof of Theorem \ref{th2.1}, for any   $\eta\in\left(0,\min\{g'(0),h'(0)\}\right)$, we can define
\begin{equation}\label{04.9}
c_r^*(\eta)\triangleq \inf\limits_{\lambda\in\mathbb R^+}\{c_\eta(\lambda)\}~~\text{and}~~c_l^*(\eta)\triangleq \sup\limits_{\lambda\in\mathbb R^-}\{c_\eta(\lambda)\}.
\end{equation}
It follows that $c_l^*<c_l^*(\eta)<c_r^*(\eta)<c_r^*$. Moreover, we have that  $c_r^*(\eta)\rightarrow c_r^*$ and $c_l^*(\eta)\rightarrow c_l^*$ as $\eta\rightarrow0$. Then for any $\epsilon>0$ small enough, there are  two small constants $\eta_1, \eta_2\in\left(0,\min\{g'(0),~h'(0)\}\right)$ such that
$
c_r^*(\eta_1)=c_r^*-\epsilon,~~c_l^*(\eta_2)=c_l^*+\epsilon
$
and
\[
\alpha\beta<(h'(0)-\eta_1)(g'(0)-\eta_1),~~\alpha\beta<(h'(0)-\eta_2)(g'(0)-\eta_2).
\]
For short,  we denote
\[
g_1\triangleq g'(0)-\eta_1,~~h_1\triangleq h'(0)-\eta_1,~~g_2\triangleq g'(0)-\eta_2,~~h_2\triangleq h'(0)-\eta_2.
\]

The following lemma gives some properties  of functions $G(c,\lambda)$ and $H(c,\lambda)$.

\begin{lemma}\label{th4.2}
For any  $c_1\in(c_r^*-\epsilon,c_r^*)$ with  $\epsilon>0$ small enough, there are  two unique constants $\zeta_1(c_1)>\gamma_1(c_1)>0$ {\rm(}denoted  also by $\zeta_1$ and $\gamma_1$ for short{\rm)} such that
\[
G(c_1,\gamma_1)H(c_1,\gamma_1)=G(c_1,\zeta_1)H(c_1,\zeta_1)=g_1 h_1
\]
and
\[
G(c_1,\rho)H(c_1,\rho)>g_1h_1,~~G(c_1,\rho)>0,~~H(c_1,\rho)>0~~ \text{for all}~\rho\in(\gamma_1,\zeta_1).
\]
Similarly, for any  $c_2\in(c_l^*,c_l^*+\epsilon)$ with  $\epsilon>0$ small enough, there are  two unique  constants $\zeta_2(c_2)<\gamma_2(c_2)<0$ {\rm(}denoted also by $\zeta_2$ and $\gamma_2$ for short{\rm)} such that
\[
G(c_2,\gamma_2)H(c_2,\gamma_2)=G(c_2,\zeta_2)H(c_2,\zeta_2)=g_2h_2
\]
and
\[
G(c_2,\rho)H(c_2,\rho)>g_2h_2,~~G(c_2,\rho)>0,~~H(c_2,\rho)>0~~\text{for all}~\rho\in(\zeta_2,\gamma_2).
\]
\end{lemma}

\begin{proof}
Similarly to the proof of  Theorem \ref{th2.1}, for any constant  $\eta\in(0,\min\{g'(0),h'(0)\})$, there are two unique constants $\lambda_r^*(\eta)\in\mathbb R^+$ and $\lambda_l^*(\eta)\in\mathbb R^-$ such that
\[
c_r^*(\eta)=c_\eta(\lambda_r^*(\eta)),~~c_l^*(\eta)= c_\eta(\lambda_l^*(\eta)),
\]
where $c_\eta(\lambda)$, $c_r^*(\eta)$ and $c_l^*(\eta)$ are defined by \eqref{04.7} and \eqref{04.9}.
Since  $\lambda_r^*(\eta_1)>0$ and
\[
\frac{\partial}{\partial c}G(c,\lambda)=\frac{\partial}{\partial c}H(c,\lambda)=\lambda,~
~c_1>c_r^*-\epsilon=c_r^*(\eta_1)=c_{\eta_1}(\lambda_r^*(\eta_1)),
\]
we get that
\begin{equation}\label{05.17}
G(c_1,\lambda_r^*(\eta_1))>G(c_{\eta_1}(\lambda_r^*(\eta_1)),\lambda_r^*(\eta_1))>0,
\end{equation}
\begin{equation}\label{05.18}
H(c_1,\lambda_r^*(\eta_1))>H(c_{\eta_1}(\lambda_r^*(\eta_1)),\lambda_r^*(\eta_1))>0.
\end{equation}
Then \eqref{04.8} implies
\[
G(c_1,\lambda_r^*(\eta_1))H(c_1,\lambda_r^*(\eta_1))>G(c_{\eta_1}(\lambda_r^*(\eta_1)),\lambda_r^*(\eta_1))H(c_{\eta_1}(\lambda_r^*(\eta_1)),\lambda_r^*(\eta_1))=g_1h_1.
\]

On the other hand, we easily get
\[
G(c_1,0)=\alpha>0,~~~H(c_1,0)=\beta>0,~~~G(c_1,0)H(c_1,0)=\alpha\beta<g_1h_1.
\]
Since
\[
G(c_1,+\infty)<0,~~~H(c_1,+\infty)<0,~~~\frac{\partial^2}{\partial \lambda^2}G(c_1,\lambda)<0,~~~\frac{\partial^2}{\partial \lambda^2}H(c_1,\lambda)<0,
\]
from \eqref{05.17} and \eqref{05.18}, there is  a unique constant $\lambda_1$ in $(\lambda_r^*(\eta_1),+\infty)$ such that
$
G(c_1,\lambda)>0,~~H(c_1,\lambda)>0~~\text{for}~~\lambda\in(0,\lambda_1)
$
and either $G(c_1,\lambda_1)=0$ or $H(c_1,\lambda_1)=0$.
Then it follows that
\[
G(c_1,\lambda_1)H(c_1,\lambda_1)=0<g_1h_1.
\]

By the arguments above, there are two  constants $\gamma_1\in(0,\lambda_r^*(\eta_1))$ and $\zeta_1\in(\lambda_r^*(\eta_1),\lambda_1)$ such that
$
G(c_1,\gamma_1)H(c_1,\gamma_1)=G(c_1,\zeta_1)H(c_1,\zeta_1)=g_1 h_1.
$
Moreover, if the constant $\lambda_0\in(0,\lambda_1)$  satisfies
\[
\frac{\partial}{\partial\lambda}\big(G(c_1,\lambda)H(c_1,\lambda)\big)\Big|_{\lambda=\lambda_0}=G(c_1,\lambda_0)\frac{\partial}{\partial \lambda}H(c_1,\lambda_0)+H(c_1,\lambda_0)\frac{\partial}{\partial \lambda}G(c_1,\lambda_0)=0,
\]
then we can get
\[
\begin{aligned}
&\frac{\partial^2}{\partial\lambda^2}\big(G(c_1,\lambda)H(c_1,\lambda)\big)\Big|_{\lambda=\lambda_0}\\
&\quad =G(c_1,\lambda_0)\frac{\partial^2}{\partial \lambda^2}H(c_1,\lambda_0)+H(c_1,\lambda_0)\frac{\partial^2}{\partial \lambda^2}G(c_1,\lambda_0)+2\frac{\partial}{\partial\lambda}G(c_1,\lambda_0)\frac{\partial}{\partial \lambda}H(c_1,\lambda_0)\\
&\quad < 0.
\end{aligned}
\]
Therefore, we have that $\gamma_1$ and $\zeta_1$ are unique and
\[
G(c_1,\rho)>0,~~H(c_1,\rho)>0,~~ G(c_1,\rho)H(c_1,\rho)>g_1h_1~~~\text{for}~\rho\in(\gamma_1,\zeta_1).
\]
Similarly, we can get the   results about $\zeta_2$ and $\gamma_2$.
\end{proof}

Now we choose some constants $\rho_1\in(\gamma_1,\zeta_1)$, $\rho_2\in(\zeta_2,\gamma_2)$,  $\delta_1>0$ and $\delta_2>0$  such that
\begin{equation}\label{3.99}
\gamma_1<\rho_1(1-\delta_1)<\rho_1(1+\delta_1)<\zeta_1,~~\zeta_2<\rho_2(1+\delta_2)<\rho_2(1-\delta_2)<\gamma_2.
\end{equation}
Then for short, we denote
\[
\begin{aligned}
&G_1^0\triangleq G(c_1,\rho_1), &&G_1^+\triangleq G(c_1,\rho_1(1+\delta_1)), && G_1^-\triangleq G(c_1,\rho_1(1-\delta_1)),\\
&H_1^0\triangleq H(c_1,\rho_1), &&H_1^+\triangleq H(c_1,\rho_1(1+\delta_1)), &&H_1^-\triangleq H(c_1,\rho_1(1-\delta_1)),\\
&\Delta_1^0=G_1^0H_1^0-g_1h_1>0, &&\Delta_1^+=G_1^+H_1^+-g_1h_1>0, &&\Delta_1^-=G_1^-H_1^--g_1h_1>0
\end{aligned}
\]
and
\[
\begin{aligned}
&G_2^0\triangleq G(c_2,\rho_2),&& G_2^+\triangleq G(c_2,\rho_2(1+\delta_2)), && G_2^-\triangleq G(c_2,\rho_2(1-\delta_2)),\\
&H_2^0\triangleq H(c_2,\rho_2), && H_2^+\triangleq H(c_2,\rho_2(1+\delta_2)), && H_2^-\triangleq H(c_2,\rho_2(1-\delta_2)),\\
&\Delta_2^0=G_2^0H_2^0-g_2h_2>0, && \Delta_2^+=G_2^+H_2^+-g_2h_2>0, && \Delta_2^-=G_2^-H_2^--g_2h_2>0.
\end{aligned}
\]
It follows from Lemma \ref{th4.2} that $G_i^0H_i^0>g_ih_i$ for each $i\in\{1,2\}$.  Therefore, we can choose some constant $\kappa_i>0$ such that
\[
\frac{g_i}{H_i^0}<\kappa_i<\frac{G_i^0}{h_i}~~\text{for each}~~i\in\{1,2\}.
\]
Since
\[
G_i^+\rightarrow G_i^0,~H_i^+\rightarrow H_i^0,~G_i^-\rightarrow G_i^0,~H_i^-\rightarrow H_i^0~~\text{as}~\delta_i\rightarrow 0^+,
\]
we can retake $\delta_i$ small enough such that the constant $\kappa_i$ also satisfies
\begin{equation}\label{03.50}
\frac{g_i}{H_i^+}<\kappa_i<\frac{G_i^+}{h_i}~~\text{and}~~\frac{g_i}{H_i^-}<\kappa_i<\frac{G_i^-}{h_i}~~\text{for each}~~i\in\{1,2\}.
\end{equation}

\begin{remark}\rm
All the notations defined in this section with  subscript ``$1$'' will be used to construct the first lower solutions spreading at a speed  of $c_1\in(c_r^*-\epsilon,c_r^*)$; Meanwhile, all the notations with  subscript ``$2$'' will be used to construct the second lower solutions spreading at a speed of $c_2\in(c_l^*,c_l^*+\epsilon)$.
\end{remark}

In addition,  we also define an auxiliary function and give its properties  in the following lemma.

\begin{lemma}\label{th4.3}
Let $M$, $N$ and $L$ be three positive  constants. For any $\delta\in(0,1)$, define
\[
f(y)=My-Ny^{1+\delta}-Ly^{1-\delta}~~\text{for}~y>0.
\]
Then we have the following conclusions
\begin{itemize}
  \item[(i)]  $F^{\rm{max}}>0$ when $M^2>4LN$, and $F^{\rm{max}}=0$ when $M^2\leqslant4LN$,
  \item[(ii)] $F^{\rm{max}}\rightarrow 0^+$ and $S-R\rightarrow0^+$ as $M^2-4LN\rightarrow 0^+$,
\end{itemize}
where
\[
F^{\rm{max}}\triangleq\sup\limits_{y>0}\big\{f(y)\big\} ~~\text{and}~~
(R,S)\triangleq\big\{y>0~|~f(y)>0\big\}~~\text{when}~M^2>4LN.
\]
\end{lemma}

\begin{proof}
Let $y_0$ and $y_1$ denote two constants satisfying
\[
y_0=\bigg[\frac{M+\sqrt{M^2-4LN(1-\delta^2)}}{2(1+\delta)N}\bigg]^{\frac{1}{\delta}},~~~y_1=\bigg[\frac{M-\sqrt{M^2-4LN(1-\delta^2)}}{2(1+\delta)N}\bigg]^{\frac{1}{\delta}}.
\]
Then we have that
\[
f'(y)
\left\{
\begin{aligned}
&<0~~~~\text{for}~y\in(0,y_1)\cup(y_0,+\infty),\\
&=0~~~~\text{for}~y=y_0~\text{and}~y=y_1,\\
&>0~~~~\text{for}~y\in(y_1,y_0)
\end{aligned}
\right.
\]
and
\[
F^{\rm{max}}\triangleq\sup\limits_{y>0}\big\{f(y)\big\}=\max\{0,~f(y_0)\}.
\]

For the fixed positive constants $M$ and $N$, we define a function
\[
F(L)\triangleq f(y_0)=My_0-Ny_0^{1+\delta}-Ly_0^{1-\delta}~~\text{for}~L>0.
\]
From some simple calculations, we get
\[
F'(L)=f'(y_0)\frac{\partial y_0}{\partial L}-y_0^{1-\delta}=-y_0^{1-\delta}<0.
\]
Notice $F(L)=0$ when $L=\frac{M^2}{4N}$. Then it follows that
\[
F(L)>0~\text{when}~L<\frac{M^2}{4N},~~
F(L)<0~\text{when}~L>\frac{M^2}{4N}.
\]
Therefore, we prove that
\[
F^{\rm{max}}>0~\text{when}~M^2>4LN,~\text{and}~F^{\rm{max}}=0~\text{when}~ M^2\leqslant4LN
\]
and
$
F^{\rm{max}}\rightarrow 0^+~~\text{as}~~M^2-4LN\rightarrow 0^+.
$
Since $(R,S)\triangleq\big\{y>0~|~f(y)>0\big\}$ when $M^2>4LN$, some simple calculations imply that
\[
R=\bigg[\frac{M-\sqrt{M^2-4LN}}{2 N}\bigg]^{\frac{1}{\delta}},~~ S=\bigg[\frac{M+\sqrt{M^2-4LN}}{2 N}\bigg]^{\frac{1}{\delta}}.
\]
Then it follows that
$
S-R\rightarrow0^+~~\text{as}~~M^2-4LN\rightarrow 0^+.
$
This completes the proof.
\end{proof}

\subsection{Lower bounds of spatial propagation}
\noindent

In this part, we  prove the lower bounds of the spatial propagation  in Theorem \ref{th2.3}.
First, we give a new method to construct lower solutions. Let $P$ denote some positive  constant satisfying that for each $i\in\{1,2\}$,
\begin{equation}\label{04.10}
P>\max\bigg\{\Big(\frac{1}{\kappa_i}\Big)^{1+\delta_i}\Big[\frac{2(G_i^0-h_i\kappa_i)^2}{G_i^--h_i\kappa_i}-(G_i^+-h_i\kappa_i)\Big],
~\frac{2 (H_i^0\kappa_i-g_i)^2}{H_i^-\kappa_i-g_i}-(H_i^+\kappa_i-g_i)\bigg\},
\end{equation}
where $g_i= g'(0)-\eta_i$, $h_i=h'(0)-\eta_i$ and $\kappa_i$ satisfies \eqref{03.50}.
Since $g$ and $h$ are in the function space $C^1[0,1]$, there is  some constant $q_0\in(0,1)$ such that for each $i\in\{1,2\}$,
\[
g(u)\geqslant (g'(0)-\frac{\eta_i}{2})u~~\text{for}~u\in(0,q_0),~~h(v)\geqslant (h'(0)-\frac{\eta_i}{2})v~~\text{for}~v\in(0,q_0).
\]
By taking $q_0$ smaller such that $q_0\leqslant\min\big\{\big(\frac{\eta_1}{2P})^{-\delta_1},~\big(\frac{\eta_2}{2P})^{-\delta_2}\big\}$, we can get
\begin{equation}\label{04.11}
g(u)\geqslant g_iu+Pu^{1+\delta_i}~~\text{for}~u\in(0,q_0),~~h(v)\geqslant h_iv+Pv^{1+\delta_i}~~\text{for}~v\in(0,q_0).
\end{equation}

Define two sets of lower solutions as follows
\begin{equation}\label{03.51}
\left\{
\begin{aligned}
&\underline u_i(t,x;\xi_i)=\max\big\{0,~f_i(e^{\rho_i(-x+c_it+\xi_i)})\big\},\\
&\underline v_i(t,x;\xi_i)=\max\big\{0,~\kappa_if_i(e^{\rho_i(-x+c_it+\xi_i)})\big\}
\end{aligned}
\right.~~\text{for each}~i\in\{1,2\},
\end{equation}
where
$
f_i(y)=y-y^{1+\delta_i}-L_iy^{1-\delta_i} ~~\text{for}~y\in\mathbb R^+,
$
and $\rho_i$, $\delta_i$ are two constants satisfying  \eqref{3.99}.
Here $L_i$ is some constant in $[\frac{1}{8},\frac{1}{4})$ and $\xi_i\in\mathbb R$ is a parameter number, and both will be chosen later. Moreover, we define
\[
\begin{aligned}
R_i=\Big[\frac{1-\sqrt{1-4L_i}}{2}\Big]^{\frac{1}{\delta_i}},~
S_i=\Big[\frac{1+\sqrt{1-4L_i}}{2}\Big]^{\frac{1}{\delta_i}},~
Y_i= \bigg[\frac{1+\sqrt{1-4L_i(1-\delta_i^2)}}{2(1+\delta_i)}\bigg]^{\frac{1}{\delta_i}}.
\end{aligned}
\]
Then Lemma \ref{th4.3} shows that
\[
\big(R_i,S_i\big)=\big\{y>0~|~f_i(y)>0\big\},~~Y_i\in\big(R_i,S_i\big),~~
F_i^{\text{max}}\triangleq \sup\limits_{y>0}\{f_i(y)\}=f_i(Y_i)>0.
\]
Also from Lemma \ref{th4.3}, we can take $L_i$  close  enough to $\frac{1}{4}$ such that
\[
\max\{F_i^{\text{max}},\kappa_iF_i^{\text{max}}\}\leqslant q_0.
\]
Therefore, we obtain from  some simple  calculations that
\[
\left\{
\begin{aligned}
&\underline u_i(t,x;\xi_i)=\underline v_i(t,x;\xi_i)=0 &&\text{for}~x-c_it\notin \Omega_i,\\
&\underline u_i(t,x;\xi_i)=\frac{1}{\kappa_i}\underline v_i(t,x;\xi_i)=f_i(e^{\rho_i(-x+c_it+\xi_i)})\in(0,F_i^{\text{max}}]&&\text{for}~x-c_it\in\Omega_i,
\end{aligned}
\right.
\]
where
$
\Omega_i=(\xi_i-\rho_i^{-1}\ln S_i,~\xi_i-\rho_i^{-1}\ln R_i).
$

Next, we prove that the pair of functions  $(\underline u_i(t,x;\xi_i),\underline v_i(t,x;\xi_i))$ is a lower solution of system \eqref{1.1} for all $\xi_i\in\mathbb R$. When $x-c_it\notin \overline\Omega_i$, we have $\underline u_i(t,x;\xi_i)=\underline v_i(t,x;\xi_i)=0$ and
\[
\begin{aligned}
&\frac{\partial}{\partial t}\underline u_i(t,x;\xi_i)-k_1*\underline u_i(t,x;\xi_i)+\underline u_i(t,x;\xi_i)+\alpha \underline u_i(t,x;\xi_i)-h(\underline v_i(t,x;\xi_i))\leqslant 0,\\
&\frac{\partial}{\partial t}\underline v_i(t,x;\xi_i)-k_2*\underline v_i(t,x;\xi_i)+\underline v_i(t,x;\xi_i)+\beta \underline v_i(t,x;\xi_i)-g(\underline u_i(t,x;\xi_i))\leqslant 0.
\end{aligned}
\]
When $x-c_it\in \overline\Omega_i$, we have    $\underline u_i(t,x;\xi_i)=\frac{1}{\kappa_i}\underline v_i(t,x;\xi_i)=f_i(e^{\rho_i(-x+c_it+\xi_i)})$. Then it follows from \eqref{04.11}  that
\[
\begin{aligned}
&\frac{\partial}{\partial t}\underline u_i(t,x;\xi_i)-k_1*\underline u_i(t,x;\xi_i)+\underline u_i(t,x;\xi_i)+\alpha \underline u_i(t,x;\xi_i)-h(\underline v_i(t,x;\xi_i))\\
& \quad \leqslant (G_i^0-h_i\kappa_i)e^{\rho_i(-x+c_it+\xi_i)}-(G_i^+-h_i\kappa_i+P\kappa_i^{1+\delta_i})e^{\rho_i(1+\delta_i)(-x+c_it+\xi_i)}\\
&\quad\quad -(G_i^--h_i\kappa_i)L_ie^{\rho_i(1-\delta_i)(-x+c_it+\xi_i)}\\
\end{aligned}
\]
and
\[
\begin{aligned}
&\frac{\partial}{\partial t}\underline v_i(t,x;\xi_i)-k_2*\underline v_i(t,x;\xi_i)+\underline v_i(t,x;\xi_i)+\beta \underline v_i(t,x;\xi_i)-g(\underline u_i(t,x;\xi_i))\\
&\quad \leqslant (H_i^0\kappa_i-g_i)e^{\rho_i(-x+c_it+\xi_i)}-(H_i^+\kappa_i-g_i+P)e^{\rho_i(1+\delta_i)(-x+c_it+\xi_i)}\\
&\quad\quad -(H_i^-\kappa_i-g_i)L_ie^{\rho_i(1-\delta_i)(-x+c_it+\xi_i)}.\\
\end{aligned}
\]
From \eqref{04.10} and $L_i>\frac{1}{8}$, we have that
\[
\begin{aligned}
&(G_i^0-h_i\kappa_i)^2-4(G_i^+-h_i\kappa_i+P\kappa_i^{1+\delta_i})(G_i^--h_i\kappa_i)L_i<(G_i^0-h_i\kappa_i)^2(1-8L_i)<0,\\
&(H_i^0\kappa_i-g_i)^2-4(H_i^+\kappa_i-g_i+P)(H_i^-\kappa_i-g_i)L_i<(H_i^0\kappa_i-g_i)^2(1-8L_i)<0.
\end{aligned}
\]
Then Lemma \ref{th4.3} shows that when $x-c_it\in \overline\Omega_i$,
\[
\begin{aligned}
&\frac{\partial}{\partial t}\underline u_i(t,x;\xi_i)-k_1*\underline u_i(t,x;\xi_i)+\underline u_i(t,x;\xi_i)+\alpha \underline u_i(t,x;\xi_i)-h(\underline v_i(t,x;\xi_i))\leqslant 0,\\
&\frac{\partial}{\partial t}\underline v_i(t,x;\xi_i)-k_2*\underline v_i(t,x;\xi_i)+\underline v_i(t,x;\xi_i)+\beta \underline v_i(t,x;\xi_i)-g(\underline u_i(t,x;\xi_i))\leqslant 0.
\end{aligned}
\]
Therefore, the pair of functions  $(\underline u_i(t,x;\xi_i),\underline v_i(t,x;\xi_i))$ is a lower solution   for any $\xi_i\in\mathbb R$.

Finally, we are ready to prove  the lower bounds    of the spatial propagation in Theorem \ref{th2.3}. The ``forward-backward spreading'' method will be applied here.

\begin{proof}[\textbf{Proof of Theorem \ref{th2.3} (lower bounds)}]

From the assumptions in  Theorem  \ref{th2.3}, we have that $u_0(x_0)>0$ and $v_0(x_0)>0$ for some constant $x_0\in\mathbb R$.
By translating the $x$-axis, we can simply suppose that $x_0=0$.
Then there are two constants $q_1>0$ and $d>0$ such that
\begin{equation}\label{04.12}
u_0(x)\geqslant q_1,~v_0(x)\geqslant q_1~~~\text{for}~~x\in[-d,d].
\end{equation}
Now  we prove that for any small $\epsilon>0$ there is some constant $\nu\in(0,1)$ such that the solution $(u(t,x),v(t,x))$ of system \eqref{1.1} satisfies
\[
(u(T,X),v(T,X))\geqslant (\nu,\nu)~~\text{for all}~T>0,~X\in[c_2T,c_1T],
\]
where $c_1\in(c_r^*-\epsilon,c_r^*)$ and $c_2\in(c_l^*,c_l^*+\epsilon)$.
For any given $T>0$ and $X\in[c_2T,c_1T]$, we denote
\[
\mu=\frac{X-c_2T}{c_1T-c_2T}\in[0,1].
\]

First, we construct a set of lower solutions in the first time period  $[0,\mu T]$ as follows
\[
\left\{
\begin{aligned}
&\underline u_1(t,x;\xi_1)=\max\big\{0,~f_1(e^{\rho_1(-x+c_1t+\xi_1)})\big\},\\
&\underline v_1(t,x;\xi_1)=\max\big\{0,~\kappa_1f_1(e^{\rho_1(-x+c_1t+\xi_1)})\big\},
\end{aligned}
\right.~~\text{for}~t\in[0,\mu T],~x\in\mathbb R,
\]
where
$
\xi_1\in [-d/2+\rho_1^{-1}\ln R_1,~d/2+\rho_1^{-1}\ln S_1]
$
and $L_1$ is some constant in  $\left[\frac{1}{8},\frac{1}{4}\right)$, which is  close  to $\frac{1}{4}$ such that
\[
\max\left\{F_1^{\text{max}},\kappa_1F_1^{\text{max}}\right\}\leqslant\min\{q_0,q_1\}~~\text{and}~~\rho_1^{-1}(\ln S_1-\ln R_1)\leqslant d/2.
\]
Then it follows  that
\[
\left\{
\begin{aligned}
&\underline u_1(t,x;\xi_1)=\underline v_1(t,x;\xi_1)=0 &&\text{for}~x-c_1t\notin \Omega_1,\\
&\underline u_1(t,x;\xi_1)=\frac{1}{\kappa_1}\underline v_1(t,x;\xi_1)=f_1(e^{\rho_1(-x+c_1t+\xi_1)})>0&&\text{for}~x-c_1t\in\Omega_1
\end{aligned}
\right.
\]
with
\begin{equation}\label{04.13}
\Omega_1=(\xi_1-\rho_1^{-1}\ln S_1,~\xi_1-\rho_1^{-1}\ln R_1)\subseteq (-d,d).
\end{equation}

From the discussion above,   the pair of functions $(\underline u_1(t,x;\xi_1),\underline v_1(t,x;\xi_1))$ is a lower solution of system \eqref{1.1}.
Moreover, we obtain that
\[
u_1(t,x;\xi_1)\leqslant F_1^{\text{max}}\leqslant q_1,~~v_1(t,x;\xi_1)\leqslant \kappa_1F_1^{\text{max}}\leqslant q_1~~\text{for}~t\geqslant0,~x\in\mathbb R.
\]
It follows from \eqref{04.12} and \eqref{04.13} that for every $\xi_1\in[-d/2+\rho_1^{-1}\ln R_1,~d/2+\rho_1^{-1}\ln S_1]$,
\[
u_0(x)\geqslant\underline u_1(0,x;\xi_1),~~v_0(x)\geqslant\underline v_1(0,x;\xi_1),~~x\in\mathbb R.
\]
Therefore,  by Lemma \ref{th4.1} we have
\[
u(t,x)\geqslant \underline u_1(t,x;\xi_1),~v(t,x)\geqslant \underline v_1(t,x;\xi_1)~~\text{for}~t\in[0,\mu T],~x\in\mathbb R.
\]
If we denote  $z_1(t)=c_1t+\xi_1-\rho_1^{-1}\ln Y_1$ for $t\in [0,\mu T]$, then
\[
\begin{aligned}
&u(t,z_1(t))\geqslant \underline u_1(t,z_1(t);\xi_1)=f_1(Y_1)=F_1^{\text{max}},\\
&v(t,z_1(t))\geqslant \underline v_1(t,z_1(t);\xi_1)=\kappa_1 f_1(Y_1)=\kappa_1F_1^{\text{max}}.
\end{aligned}
\]
Furthermore, the arbitrariness of $\xi_1$  and $R_1<Y_1<S_1$ show that
\[
u(t,x)\geqslant F_1^{\text{max}},~v(t,x)\geqslant \kappa_1F_1^{\text{max}}~~\text{for all}~t\in[0,\mu T],~x\in[c_1t-d/2,~c_1t+d/2].
\]
Therefore, there is some constant $q_2=\min\{ F_1^{\text{max}},\kappa_1F_1^{\text{max}}\}$ such that
\begin{equation}\label{04.14}
u(\mu T,x)\geqslant q_2,~v(\mu T,x)\geqslant q_2~~\text{for}~x\in[c_1\mu T-d/2,~c_1\mu T+d/2].
\end{equation}

Next, we construct another set of lower solutions in the second  time period $[\mu T,T]$ as follows
\[
\left\{
\begin{aligned}
&\underline u_2(t,x;\xi_2)=\max\big\{0,~f_2(e^{\rho_2(-x+c_2t+\xi_2)})\big\},\\
&\underline v_2(t,x;\xi_2)=\max\big\{0,~\kappa_2f_2(e^{\rho_2(-x+c_2t+\xi_2)})\big\},
\end{aligned}
\right.~~\text{for}~t\in[\mu T,T],~x\in\mathbb R,
\]
where
$
\xi_2\in \big[(c_1-c_2)\mu T+\rho_2^{-1}\ln R_2,~(c_1-c_2)\mu T+\rho_2^{-1}\ln S_2\big]
$
and $L_2$ is some constant in  $[\frac{1}{8},\frac{1}{4})$, which is  close  to $\frac{1}{4}$ such that
\[
\max\big\{F_2^{\text{max}},\kappa_2F_2^{\text{max}}\big\}\leqslant q_2 ~~\text{and}~~\rho_2^{-1}(\ln S_2-\ln R_2)\leqslant d/2.
\]
Then it follows that
\[
\left\{
\begin{aligned}
&\underline u_2(t,x;\xi_2)=\underline v_2(t,x;\xi_2)=0 &&\text{for}~x-c_2t\notin \Omega_2,\\
&\underline u_2(t,x;\xi_2)=\frac{1}{\kappa_2}\underline v_2(t,x;\xi_2)=f_2(e^{\rho_2(-x+c_2t+\xi_2)})>0&&\text{for}~x-c_2t\in\Omega_2
\end{aligned}
\right.
\]
with
$
\Omega_2=(\xi_2-\rho_2^{-1}\ln S_2,~\xi_2-\rho_2^{-1}\ln R_2).
$

As stated above, the pair of functions $(\underline u_2(t,x;\xi_2),\underline v_2(t,x;\xi_2))$ is also a lower solution of system \eqref{1.1}.
At the time $t=\mu T$, we have that
\[
\left\{
\begin{aligned}
&\underline u_2(\mu T,x;\xi_2)=\underline v_2(\mu T,x;\xi_2)=0 &&\text{for}~x\notin c_2\mu T+\Omega_2,\\
&\underline u_2(\mu T,x;\xi_2)=\frac{1}{\kappa_2}\underline v_2(\mu T,x;\xi_2)\in(0,q_2)&&\text{for}~x \in c_2\mu T+\Omega_2,
\end{aligned}
\right.
\]
where
\[
c_2\mu T+\Omega_2\triangleq  (c_2\mu T+\xi_2-\rho_2^{-1}\ln S_2,~c_2\mu T+\xi_2-\rho_2^{-1}\ln R_2).
\]
It follows that $c_2\mu T+\Omega_2\subseteq (c_1\mu T-d/2,c_1\mu T-d/2)$.
Then we get from \eqref{04.14} that for every $\xi_2\in \big[(c_1-c_2)\mu T+\rho_2^{-1}\ln R_2,~(c_1-c_2)\mu T+\rho_2^{-1}\ln S_2\big]$,
\[
u(\mu T,x)\geqslant \underline u_2(\mu T,x;\xi_2),~v(\mu T,x)\geqslant \underline v_2(\mu T,x;\xi_2),~~x\in \mathbb R.
\]
Therefore, Lemma \ref{th4.1} implies that
\[
u(t,x)\geqslant \underline u_2(t,x;\xi_2),~v(t,x)\geqslant \underline v_2(t,x;\xi_2)~~\text{for}~t\in[\mu T, T],~x\in\mathbb R.
\]

If we denote $z_2(t)=c_2t+\xi_2-\rho_2^{-1}\ln Y_2$ for $t\in[\mu T, T]$, then
\[
\begin{aligned}
&u(t,z_2(t))\geqslant \underline u_1(t,z_2(t);\xi_2)=f_2(Y_2)=F_2^{\text{max}},\\
&v(t,z_2(t))\geqslant \underline v_1(t,z_2(t);\xi_2)=\kappa_2 f_2(Y_2)=\kappa_2F_2^{\text{max}}.
\end{aligned}
\]
Furthermore, the  arbitrariness of   $\xi_2$  and $R_2<Y_2<S_2$ show  that
\[
u(t,x)\geqslant F_2^{\text{max}},~v(t,x)\geqslant  \kappa_2F_2^{\text{max}}~~\text{for all}~t\in[\mu T, T],~x=c_2t+(c_1-c_2)\mu T.
\]
By taking $\nu=\min\{F_2^{\text{max}},~\kappa_2F_2^{\text{max}}\}$,  we get from $X=c_2T+(c_1-c_2)\mu T$  that
\[
u(T,X)\geqslant \nu,~v(T,X)\geqslant \nu~~\text{for}~T>0,~X\in[c_2T,c_1T].
\]
Therefore, for any small constant $\epsilon>0$ we have that
\[
\inf\limits_{(c_l^*+\epsilon)t \leqslant x\leqslant (c_r^*-\epsilon)t}(u(t, x),v(t,x))\geqslant(\nu, \nu)~~\text{for}~t>0.
\]
This completes the proof.
\end{proof}

\subsection{Upper bounds of spatial propagation}
\noindent

\begin{proof}[\textbf{Proof of Theorem \ref{th2.3} (upper bounds)}]

In this subsection, we prove that
\begin{equation}\label{04.15}
\sup\limits_{x\leqslant (c_l^*-\epsilon)t}(u(t, x),v(t,x))\rightarrow(0,0)~~\text{and}~\sup\limits_{x\geqslant (c_r^*+\epsilon)t}(u(t, x),v(t,x))\rightarrow(0,0)~\text{as}~t\rightarrow+\infty.
\end{equation}
First, we define the functions
\begin{equation}\label{04.99}
\left\{
\begin{aligned}
&\bar u(t,x)=\min\Big\{1,~\Gamma e^{\lambda_l^*(-x+c_l^*t)},~\Gamma e^{\lambda_r^*(-x+c_r^*t)}\Big\},\\
&\bar v(t,x)=\min\Big\{1,~b(\lambda_l^*)\Gamma e^{\lambda_l^*(-x+c_l^*t)},~b(\lambda_r^*)\Gamma e^{\lambda_r^*(-x+c_r^*t)}\Big\}
\end{aligned}
\right.
\end{equation}
for $t\geqslant0$ and $x\in\mathbb R$, where the function $b(\lambda)$ is defined by \eqref{04.5}. From the assumptions in Theorem \ref{th2.3}, we can take $\Gamma$ large enough such that  $\Gamma\geqslant \max\big\{1,~\Gamma_0,~\frac{1}{b(\lambda_l^*)},~\frac{1}{b(\lambda_r^*)}\big\}$ and
\begin{equation}\label{04.16}
\bar u(0,x)\geqslant u_0(x),~~\bar v(0,x)\geqslant v_0(x)~~\text{for}~~x\in\mathbb R.
\end{equation}

Next, we prove that the pair of  functions $(\bar u(t,x), \bar v(t,x))$ is an upper solution of system \eqref{1.1}.
When $x\leqslant c^*_lt+(\lambda_l^*)^{-1}\ln \Gamma$, we have that $\bar u(t,x)=\Gamma e^{\lambda_l^*(-x+c_l^*t)}$ and  $\bar v(t,x)\leqslant b(\lambda_l^*)\Gamma e^{\lambda_l^*(-x+c_l^*t)}$.
Then it follows from (H2) and \eqref{04.6} that
\[
\partial_t \bar u-k_1*\bar u+\bar u+\alpha \bar u-h(\bar v)\geqslant \big[G(c_l^*,\lambda_l^*)-h'(0)b(\lambda_l^*)\big]\Gamma e^{\lambda_l^*(-x+c_l^*t)}=0.
\]
Similarly, when $x\geqslant c^*_rt+(\lambda_r^*)^{-1}\ln \Gamma$, we get from (H2) and \eqref{04.6}  that
\[
\partial_t \bar u-k_1*\bar u+\bar u+\alpha \bar u-h(\bar v)\geqslant \big[G(c_r^*,\lambda_r^*)-h'(0)b(\lambda_r^*)\big]\Gamma e^{\lambda_r^*(-x+c_r^*t)}=0.
\]
If $x\in\big[c^*_lt+(\lambda_l^*)^{-1}\ln \Gamma,~c^*_rt+(\lambda_r^*)^{-1}\ln \Gamma\big]$, then $\bar u(t,x)=1$ and $\bar v(t,x)\leqslant1$,
which  implies that
\[
\partial_t \bar u-k_1*\bar u+\bar u+\alpha \bar u-h(\bar v)\geqslant\alpha-h(\bar v)\geqslant\alpha-h(1)=0.
\]
Therefore, we finally obtain that
\[
\partial_t \bar u-k_1*\bar u+\bar u+\alpha \bar u-h(\bar v)\geqslant0~~~\text{for all}~t>0,~x\in\mathbb R.
\]
Similarly, we can  obtain
\[
\partial_t \bar v-k_2*\bar v+\bar v+\beta \bar v-g(\bar u)\geqslant0~~~\text{for all}~t>0,~x\in\mathbb R.
\]

From Lemma \ref{th4.1} and \eqref{04.16}, it follows that
\[
(u(t,x),v(t,x))\leqslant(\bar u(t,x),\bar v(t,x))~~\text{for}~~t\geqslant0,~x\in\mathbb R.
\]
Then we have
\[
\begin{aligned}
&\sup\limits_{x\leqslant (c_l^*-\epsilon)t}(u(t, x),v(t,x))\leqslant\sup\limits_{x\leqslant (c_l^*-\epsilon)t}(\bar u(t, x),\bar v(t,x)) \leqslant \big(\Gamma e^{\lambda_l^*\epsilon t},~b(\lambda_l^*)\Gamma e^{\lambda_l^*\epsilon t}\big),\\
&\sup\limits_{x\geqslant (c_r^*+\epsilon)t}(u(t, x),v(t,x))\leqslant\sup\limits_{x\geqslant (c_r^*+\epsilon)t}(\bar u(t, x),\bar v(t,x))\leqslant \big(\Gamma e^{-\lambda_r^*\epsilon t},~b(\lambda_r^*)\Gamma e^{-\lambda_r^*\epsilon t}\big).
\end{aligned}
\]
Therefore, using $\lambda_l^*<0<\lambda_r^*$, we finish the proof of \eqref{04.15}.
\end{proof}

\begin{remark} \rm
The irreducibility of the linearized system at zero is a necessary property in this paper. In fact,  our idea of the new lower solution \eqref{03.51} is from the following  system
\[
\left\{\begin{aligned}
& u_t=k_1*u-u-\alpha u+(h'(0)-\eta)v+Pv^{1+\delta}, ~~t>0,~x\in\mathbb R,\\
& v_t=k_2*v-v-\beta v+(g'(0)-\eta)u+Pu^{1+\delta}, ~~~t>0,~x\in\mathbb R,
\end{aligned}
\right.
\]
where $\delta>0$ is an appropriate constant and $\eta>0$ is a  constant small enough, see the condition   \eqref{04.11}.
If the linearized system at zero is reducible (namely, $h'(0)$ or $g'(0)$ is equal to $0$), the above system becomes  non-cooperative and meanwhile Lemma \ref{th4.2}  does not hold.
Then there are not any $\rho_i$ and $\delta_i$ satisfying  \eqref{3.99}. Thus, we can not construct any lower solution  in the  form of  \eqref{03.51}.
Moreover, in some studies (for example Weinberger et al. \cite{WLL2002}) the irreducibility can be replaced by some other assumptions on the matrix in Frobenius form.

\end{remark}

\begin{remark} \rm
The linear and nonlinear selection of speed is an important problem in reaction-diffusion systems. In system \eqref{1.1}, the condition for linear selection is given by
\begin{equation}\label{06.3}
g(u) \leqslant g'(0)u~~\text{and}~~h(v)\leqslant h'(0)v.
\end{equation}
However, when \eqref{06.3} is not satisfied,  the upper solution  \eqref{04.99} becomes unavailable and thus the upper bound \eqref{04.15} of spatial propagation  is no longer right. In order to obtain the upper bound, we can use $g(u) \leqslant \hat{g}u$ and $h(v)\leqslant \hat{h}v$  instead of \eqref{06.3}, where
\[
~\hat{g}=\sup_{u\in(0,1]}\{g(u)/u\}~\text{and}~\hat{h}=\sup_{v\in(0,1]}\{h(v)/v\}.
\]
Under the same assumptions except \eqref{06.3} as in Theorem \ref{th2.3}, when $k_1$ and $k_2$ are  symmetric, we can obtain  that
\[
\left\{\begin{aligned}
&\lim\limits_{t\rightarrow+\infty}\sup\limits_{|x|\geqslant (c^++\epsilon)t}\big(u(t, x),v(t,x)\big)\rightarrow(0,0),\\
&\inf\limits_{|x|\leqslant (c^--\epsilon)t}\big(u(t, x),v(t,x)\big)\geqslant(\nu,\nu)~~\text{for all}~~t\geqslant 0,
\end{aligned}
\right.
\]
where the constants  $c^+$ and $c^-$ satisfy that $c^+\geqslant c^-$ and
\[
\begin{aligned}
&c^+\leqslant\inf_{\lambda\in\mathbb R^+}\left\{\frac{1}{2\lambda}\left[A(\lambda)+B(\lambda)+\sqrt{(A(\lambda)-B(\lambda))^2+4 \hat{h}\hat{g}}\right]\right\},\\
&c^-\geqslant\inf_{\lambda\in\mathbb R^+}\left\{\frac{1}{2\lambda}\left[A(\lambda)+B(\lambda)+\sqrt{(A(\lambda)-B(\lambda))^2+4 h'(0)g'(0)}\right]\right\}.
\end{aligned}
\]
However, it is  challenging to prove that $c^+=c^-$. For more results about the  linear and nonlinear  selection of speed, see e.g. Alhasanat and Ou\cite{AO2020},  Ma and Ou\cite{MO2019}, Ma et al. \cite{MHO2019} and Wang et al. \cite{WHO2020}.
\end{remark}

\section{Second type of initial data and symmetric kernels}
\noindent

In this section, under the assumption that $k_1$ and $k_2$ are symmetric, we prove the monotone property   and the spatial propagation result for the second type of initial data.

\subsection{Monotone property}
\noindent

The following theorem gives a monotone property result of system \eqref{1.1}.

\begin{theorem}\label{th2.6}
If $k_1(\cdot)$, $k_2(\cdot)$, $u_0(\cdot)$ and $v_0(\cdot)$ are symmetric   and decreasing on $\mathbb R^+$, so are the functions $u(t,\cdot)$ and $v(t,\cdot)$ at any time $t>0$, where $(u(t,x),v(t,x))$ is the solution of  \eqref{1.1}.
\end{theorem}

\begin{proof}
First, the symmetry properties of $u(t,\cdot)$ and $v(t,\cdot)$ can be obtained easily. Indeed, by considering the system
\[
\left\{\begin{aligned}
&\frac{\partial}{\partial t}  w_1(t,x)=k_1*w_1(t,x)-w_1(t,x)-\alpha w_1(t,x)+h(w_2(t,x)), ~~t>0,~x\in\mathbb R,\\
&\frac{\partial}{\partial t}  w_2(t,x)=k_2*w_2(t,x)-w_2(t,x)-\beta w_2(t,x)+g(w_1(t,x)), ~~~t>0,~x\in\mathbb R,\\
&w_1(0,x)=u_0(-x),~w_2(0,x)=v_0(-x),~~x\in\mathbb R
\end{aligned}
\right.
\]
and using the uniqueness property of the solution, we have $u(t,x)=w_1(t,x)=u(t,-x)$ and $v(t,x)=w_2(t,x)=v(t,-x)$ for $t\geqslant0$, $x\in\mathbb R$.

Next, we prove the monotone property.
For a fixed constant $y>0$, we define
\[
m_1(t,x)=u(t,x+2y)-u(t,x),~m_2(t,x)=v(t,x+2y)-v(t,x)~~\text{for}~t\geqslant0,~x\in\mathbb R.
\]
Then the symmetric properties of $u(t,\cdot)$ and $v(t,\cdot)$ imply that
\[
m_1(t,-y)=m_2(t,-y)=0~~\text{for}~t\geqslant0.
\]
At time $t=0$, we easily get that
\[
\begin{aligned}
&m_1(0,x)\leqslant 0,~m_2(0,x)\leqslant 0~~~\text{for}~x>-y,\\
&m_1(0,x)\geqslant 0,~m_2(0,x)\geqslant 0~~~\text{for}~x<-y.
\end{aligned}
\]
In order to show that $u(t,\cdot)$ and $v(t,\cdot)$ are decreasing in $\mathbb R^+$, we prove that
\begin{equation}\label{05.12}
m_1(t,x)\leqslant 0,~m_2(t,x)\leqslant 0~\text{for all}~t>0,~x>-y.
\end{equation}
Indeed, if \eqref{05.12} holds, then  $u(t,x+2y)\leqslant u(t,x)$ and $v(t,x+2y)\leqslant v(t,x)$ for all $x>-y$ and $t>0$, which imply that $u(t,\cdot)$ and $v(t,\cdot)$ are decreasing in $\mathbb R^+$.

Now we prove \eqref{05.12}.
Since $h(\cdot)\in C^1([0,1])$, there is some constant  $M>0$ such that for all $t\geqslant0$ and $x\in\mathbb R$,
\begin{equation}\label{05.13}
\begin{aligned}
\frac{\partial}{\partial t} m_1(t,x)&=k_1*m_1(t,x)-m_1(t,x)-\alpha m_1(t,x)+h(v(t,x+2y))-h(v(t,x))\\
&\leqslant k_1*m_1(t,x)-m_1(t,x)-\alpha m_1(t,x)+M  m_2(t,x).
\end{aligned}
\end{equation}
Now we suppose that \eqref{05.12} does not hold, which means that there are two constants $T_0>0$ and $\varepsilon>0$ such that
\begin{equation}\label{05.14}
m_1(t,x)<\varepsilon e^{Kt},~m_2(t,x)<\varepsilon e^{Kt}~\text{for all}~t\in(0,T_0),~x>-y
\end{equation}
and at least one of the following  two results holds:
\begin{equation}\label{05.15}
\sup\limits_{x>-y}\{m_1(T_0,x)\}=\varepsilon e^{KT_0},~~m_2(T_0,x)\leqslant\varepsilon e^{KT_0}~~\text{for}~~x>-y;
\end{equation}
\[
m_1(T_0,x)\leqslant\varepsilon e^{KT_0}~\text{for}~x>-y,~~~\sup\limits_{x>-y}\{m_2(T_0,x)\}=\varepsilon e^{KT_0}.
\]
Here $K$ is a positive constant satisfying $K>\frac{4}{3}(M+1)-\alpha$.   Without loss of generality, we assume    \eqref{05.15} holds.
As stated in the proof of  \cite[Lemma 2.2]{XLR2018}, when $m_1(t,x)\geqslant0$, it holds that
\begin{equation}\label{05.16}
k_1*m_1(t,x)-m_1(t,x)\leqslant \varepsilon e^{Kt}~~\text{for}~t\in(0,T_0],~x>-y.
\end{equation}
From \eqref{05.15}, at least one of the following cases must hold:
\begin{description}
  \item[Case 1] there is $x_0\in(-y,+\infty)$ such that $m_1(T_0,x_0)=\sup\limits_{x>-y}\left\{m_1(T_0,x)\right\}=\varepsilon e^{KT_0}$,
  \item[Case 2]  $\limsup\limits_{x\rightarrow+\infty}\{m_1( T_0,x)\}=\varepsilon e^{KT_0}$.
\end{description}

If Case 1 holds,   it follows that
\[
\left.\frac{\partial}{\partial t}\left(m_1(t,x_0)-\varepsilon e^{Kt}\right)\right|_{t=T_0}\geqslant 0,
\]
which means
\[
\frac{\partial}{\partial t} m_1(T_0,x_0)\geqslant\varepsilon K e^{KT_0}.
\]
Then  from \eqref{05.15} and \eqref{05.16} we get
\[
\begin{aligned}
&\frac{\partial}{\partial t} m_1(T_0,x_0)-k_1*m_1(T_0,x_0)+m_1(T_0,x_0)+\alpha m_1(T_0,x_0)-M  m_2(T_0,x_0)\\
& \quad \geqslant ( K -1+\alpha -M) \varepsilon e^{KT_0}>0.
\end{aligned}
\]
It is a contradiction to \eqref{05.13}, which implies that \eqref{05.12} holds.

If Case 2 holds,  there is some constant  $x_1$ large enough such that
\[
m_1(T_0,x_1)>\frac{3}{4}\varepsilon e^{KT_0}.
\]
For all $\sigma>0$, we define
\[
\rho_{\sigma}(t,x)= \left[\frac{1}{2}+\sigma q_0(x)\right]\varepsilon e^{Kt}~\text{for}~t\in[0,T_0],~x\in\mathbb R,
\]
where $q_0(x)$ is a smooth and  increasing function satisfying
\[
q_0(x)=\left\{
\begin{aligned}
&1~~~\text{for}~x\leqslant x_1,\\
&3~~~\text{for}~x\geqslant x_1+1.
\end{aligned}
\right.
\]
Let $\sigma^*$ be a constant denoted by
\[
\sigma^*=\inf\Big\{\sigma>0~|~m_1(t,x)-\rho_{\sigma}(t,x)\leqslant0 ~~\text{for}~t\in[0,T_0],~x>-y \Big\}.
\]
Moreover, some simple  calculations yield that $\frac{1}{4}\leqslant\sigma^*\leqslant\frac{1}{2}$ and
\[
\rho_{\sigma^*}(t,x)\geqslant\frac{5}{4}\varepsilon e^{Kt}>m_1(t,x)~~\text{for}~t\in[0,T_0],~x\geqslant x_1+1.
\]
From the definition of $\sigma^*$, there must exist $T_1\in(0,T_0]$ and $x_2\in(-y,x_1+1)$ such that
\[
m_1(T_1,x_2)-\rho_{\sigma^*}(T_1,x_2)=\sup\limits_{t\in[0,T_0],~x>-y}\big\{m_1(t,x)-\rho_{\sigma^*}(t,x)\big\}=0,
\]
Then we have that
\[
\begin{aligned}
&m_1(T_1,x_2)=\rho_{\sigma^*}(T_1,x_2)\geqslant\rho_{\frac{1}{4}}(T_1,x_2)\geqslant \frac{3}{4}\varepsilon e^{KT_1},\\
&\frac{\partial}{\partial t}m_1(T_1,x_2)\geqslant \frac{\partial}{\partial t}\rho_{\sigma^*}(T_1,x_2)=K\rho_{\sigma^*}(T_1,x_2)\geqslant K\rho_{\frac{1}{4}}(T_1,x_2)\geqslant \frac{3}{4}K\varepsilon e^{KT_1}.
\end{aligned}
\]
From \eqref{05.14} and \eqref{05.16}, it follows
\[
\begin{aligned}
&\frac{\partial}{\partial t}m_1(T_1,x_2)-k_1*m_1(T_1,x_2)+m_1(T_1,x_2)+\alpha m_1(T_1,x_2)-Mm_2(T_1,x_2)\\
&\quad \geqslant (\frac{3}{4}K-1+\frac{3}{4}\alpha-M)\varepsilon  e^{KT_1}>0,\\
\end{aligned}
\]
which contradicts \eqref{05.13}. Therefore, we finish the proof of  \eqref{05.12}.
\end{proof}

\subsection{Spatial propagation}
\noindent

In this subsection, we study  the spatial propagation of system \eqref{1.1} for the second type of initial data and symmetric kernels. The following theorem is the main result.
\begin{theorem}\label{th2.5}
Assume that {\rm(H1)} and {\rm(H2)} hold. Let $k_1$, $k_2$  satisfy {\rm(K1)} and be symmetric on $\mathbb R$ and decreasing in $\mathbb R^+$. If $u_0(\cdot)$ and $v_0(\cdot)$ are two continuous functions satisfying $0< u_0(x)\leqslant 1$, $0< v_0(x)\leqslant 1$ for $x\in\mathbb R$ and
\[
u_0(x)\sim O(e^{-\lambda |x|}),~~v_0(x)\sim O(e^{-\lambda |x|})~~\text{as}~|x|\rightarrow +\infty~\text{with}~ \lambda\in(0,\lambda^*),
\]
then for any  $\epsilon\in(0,c(\lambda))$ there is some constant $\nu\in(0,1)$ such that the solution  of system \eqref{1.1} has the following properties
\[
\left\{\begin{aligned}
&\lim\limits_{t\rightarrow+\infty}\sup\limits_{|x|\geqslant (c(\lambda)+\epsilon)t}\big(u(t, x),v(t,x)\big)\rightarrow(0,0),\\
&\inf\limits_{|x|\leqslant c(\lambda)t}\big(u(t, x),v(t,x)\big)\geqslant(\nu,\nu)~~\text{for all}~~t\geqslant 0,
\end{aligned}
\right.
\]
where $\lambda^*\triangleq\lambda_r^*=-\lambda_l^*$.
\end{theorem}

\begin{remark}
\rm
From Theorem \ref{th2.5} and the definition of $c(\lambda)$ in \eqref{02.1}, we obtain  a relationship between the spreading speeds and the  exponentially decaying rate  of initial data. Moreover,   Theorem \ref{th2.1} shows that $c'(\lambda)<0$ for all $\lambda\in(0,\lambda^*)$; namely,  the spreading speed $c(\lambda)$ is decreasing with respect to  $\lambda\in (0,\lambda^*)$. Meanwhile, we also have that $\inf\{c(\lambda)~|~\lambda\in(0,\lambda^*)\}=c^*$, which implies that the minimum value of $c(\lambda)$ coincides with the spreading speed for the first type of initial value and symmetric kernels.
\end{remark}
Before proving Theorem \ref{th2.5}, we give the following lemma.
\begin{lemma}\label{th3.1}
For any $\lambda\in(0,\lambda_r^*)$, there is a unique constant $\delta_\lambda>0$ such that
\[
c(\lambda)=c(\lambda+\lambda\delta_\lambda)~~\text{and}~~c(\eta)<c(\lambda)~~\text{for}~\eta\in(\lambda,\lambda+\lambda\delta_\lambda).
\]
Similarly, for any $\lambda\in(\lambda_l^*,0)$, there is a unique constant $\delta_\lambda>0$ such that
\[
c(\lambda)=c(\lambda+\lambda\delta_\lambda)~~\text{and}~~c(\eta)>c(\lambda)~~\text{for}~\eta\in(\lambda+\lambda\delta_\lambda,\lambda).
\]
\end{lemma}

\begin{proof}
 Since $D(\lambda)>A(\lambda)$ for all $\lambda\in\mathbb R$ and
\[
\lim\limits_{\lambda\rightarrow+\infty}\frac{A(\lambda)}{\lambda}=+\infty,~~\lim\limits_{\lambda\rightarrow-\infty}\frac{A(\lambda)}{\lambda}=-\infty,
\]
from \eqref{02.1} we get that $\lim\limits_{\lambda\rightarrow+\infty}c(\lambda)=+\infty$ and   $\lim\limits_{\lambda\rightarrow-\infty}c(\lambda)=-\infty$.
On the other hand, from $D(0)\in(0,+\infty)$ it follows that $\lim\limits_{\lambda\rightarrow0^+}c(\lambda)=+\infty$ and $\lim\limits_{\lambda\rightarrow0^-}c(\lambda)=-\infty$. Therefore, by \eqref{03.5}, we finish the proof of Lemma \ref{th3.1}.
\end{proof}

Now we are ready to prove Theorem \ref{th2.5}.

\begin{proof}[\textbf{Proof of Theorem \ref{th2.5}}]
For any $\lambda\in(0,\lambda^*)$, let $\delta_\lambda$ denote the constant in Lemma \ref{th3.1}, then
$
c(\lambda)>c\big(\lambda(1+\delta)\big)~~\text{for}~\delta\in(0,\delta_\lambda).
$
We denote $G(c,\lambda)$, $H(c,\lambda)$  and $b(\lambda)$   by \eqref{04.1}, \eqref{04.2} and \eqref{04.5},  respectively.
Since
$
\frac{\partial}{\partial c}G(c,\lambda)=\frac{\partial}{\partial c}H(c,\lambda)=\lambda\in(0,\lambda^*),
$
from \eqref{04.3} we get
\[
\left\{
\begin{aligned}
G\big(c(\lambda),\lambda(1+\delta)\big)>G\big(c\big(\lambda(1+\delta)\big),\lambda(1+\delta)\big)>0,\\
H\big(c(\lambda),\lambda(1+\delta)\big)>H\big(c\big(\lambda(1+\delta)\big),\lambda(1+\delta)\big)>0,
\end{aligned}~~~~\text{for}~\lambda\in(0,\lambda^*),~\delta\in(0,\delta_\lambda).
\right.
\]
Therefore, it follows from \eqref{04.6} that
\begin{equation}\label{05.1}
\frac{g'(0)}{H\big(c(\lambda),\lambda(1+\delta)\big)}<b\big(\lambda(1+\delta)\big)<\frac{G(c(\lambda),\lambda(1+\delta))}{h'(0)}~~\text{for}~\lambda\in(0,\lambda^*),~ \delta\in(0,\delta_\lambda).
\end{equation}

\textbf{Step 1.} Now we prove that
\begin{equation}\label{05.2}
\sup\limits_{|x|\geqslant (c(\lambda)+\epsilon)t}\big(u(t, x),v(t,x)\big)\rightarrow(0,0)~~\text{as}~t\rightarrow+\infty.
\end{equation}
For any given $\lambda\in(0,\lambda^*)$, define
\begin{equation}
\left\{
\begin{aligned}
&\bar u(t,x)=\min\big\{1,~\Gamma e^{\lambda(-|x|+c(\lambda)t)}\big\},\\
&\bar v(t,x)=\min\big\{1,~b(\lambda)\Gamma e^{\lambda(-|x|+c(\lambda)t)}\big\},
\end{aligned}
\right.~~\text{for}~t\geqslant0,~x\in\mathbb R,
\end{equation}
where  the constant $\Gamma$ is large enough such that  $\Gamma\geqslant\max\big\{1,\frac{1}{b(\lambda)}\big\}$. By the assumptions about initial data in Theorem \ref {th2.5}, we can  take $\Gamma$ larger if necessary such that
\begin{equation}\label{05.3}
\bar u(0,x)\geqslant u_0(x),~\bar v(0,x)\geqslant v_0(x)~~\text{for}~x\in\mathbb R.
\end{equation}

Now we prove that the pair of  functions  $(\bar u(t,x),\bar v(t,x))$ is an upper solution of system \eqref{1.1}.
If $|x|\leqslant c(\lambda)t+\lambda^{-1}\ln \Gamma$,  we have $\bar u(t,x)=1$ and $\bar v(t,x)\leqslant1$. Then it follows from (H1) and (H2) that
\[
\partial_t \bar u-k_1*\bar u+\bar u+\alpha \bar u-h(\bar v)\geqslant\alpha-h(\bar v)\geqslant \alpha-h(1)=0.
\]
If $|x|> c(\lambda)t+\lambda^{-1}\ln \Gamma$, we get  $\bar u(t,x)=\Gamma e^{\lambda(-|x|+c(\lambda)t)}$ and $\bar v(t,x)\leqslant b(\lambda)\Gamma e^{\lambda(-|x|+c(\lambda)t)}$.
By (H2) and \eqref{04.6}, some simple calculations imply that
\[
\partial_t \bar u-k_1*\bar u+\bar u+\alpha \bar u-h(\bar v)\geqslant \big[G(c(\lambda),\lambda)-h'(0)b(\lambda)\big]\Gamma e^{\lambda(-|x|+c(\lambda)t)}=0.
\]
We finally get that
\begin{equation}\label{05.4}
\partial_t \bar u-k_1*\bar u+\bar u+\alpha \bar u-h(\bar v)\geqslant0~~\text{for all}~t>0,~x\in\mathbb R.
\end{equation}
Meanwhile, if $|x|\leqslant c(\lambda)t+\lambda^{-1}\ln (b(\lambda)\Gamma)$, we have $\bar v(t,x)= 1$ and $\bar u(t,x)\leqslant1$. Then it follows from  (H1) and (H2) that
\[
\partial_t \bar v-k_2*\bar v+\bar v+\beta \bar v-g(\bar u)\geqslant\beta-g(\bar u)\geqslant \beta-g(1)=0.
\]
If $|x|> c(\lambda)t+\lambda^{-1}\ln (b(\lambda)\Gamma)$,  we get that  $\bar v(t,x)=b(\lambda)\Gamma e^{\lambda(-|x|+c(\lambda)t)}$ and $\bar u(t,x)\leqslant\Gamma e^{\lambda(-|x|+c(\lambda)t)}$. By (H2) and \eqref{04.6}, some simple calculations show
\[
\partial_t \bar v-k_2*\bar v+\bar v+\beta \bar v-g(\bar u)\geqslant\big[H(c(\lambda),\lambda)b(\lambda)-g'(0)\big]\Gamma e^{\lambda(-|x|+c(\lambda)t)}=0.
\]
We finally get that
\begin{equation}\label{05.5}
\partial_t \bar v-k_2*\bar v+\bar v+\beta \bar v-g(\bar u)\geqslant0~~\text{for all}~t>0,~x\in\mathbb R.
\end{equation}
Therefore, $(\bar u(t,x),\bar v(t,x))$ is an upper solution of system \eqref{1.1}.

By  \eqref{05.3}-\eqref{05.5}, Lemma \ref{th4.1} shows that
\[
(u(t,x),v(t,x))\leqslant(\bar u(t,x),\bar v(t,x))~~\text{for}~~t\geqslant0,~x\in\mathbb R.
\]
Then we have
\[
\sup\limits_{|x|\geqslant (c(\lambda)+\epsilon)t}(u(t, x),v(t, x))\leqslant\sup\limits_{|x|\geqslant (c(\lambda)+\epsilon)t}(\bar u(t, x),\bar v(t, x))\leqslant(\Gamma e^{-\lambda\epsilon t},b(\lambda)\Gamma e^{-\lambda\epsilon t}),
\]
which implies that \eqref{05.2} holds.

\textbf{Step 2.} Next we  prove that
\[
\big(u(t, x),v(t,x)\big)\geqslant(\nu,\nu)~~\text{for all}~~t\geqslant 0,~|x|\leqslant c(\lambda)t.
\]
From  the assumptions in Theorem \ref{th2.5}, there exists a continuous symmetric function $w_0(x)$, which is decreasing in $\mathbb R^+$ and satisfies that
\[
u_0(x)\geqslant w_0(x),~~v_0(x)\geqslant w_0(x)~~\text{for}~x\in\mathbb R,~~w_0(x)=\left\{
\begin{aligned}
&\gamma_0 e^{-\lambda|x|}, &&|x|\geqslant y_0,\\
&p_1\triangleq\gamma_0 e^{-\lambda y_0}, &&|x|\leqslant y_0,\\
\end{aligned}
\right.
\]
where $\gamma_0$ and $y_0$ are two positive constants.
Let $p$ and $\delta$ denote two constants satisfying $p=\min\{p_0,p_1\}$ and $0<\delta<\min\{\delta_0,\delta_\lambda\}$. Then by $g(\cdot),~h(\cdot)\in C^{1+\delta_0}\big([0,p_0]\big)$, we can find some constant $M>0$ such that
\begin{equation}\label{05.6}
g(u)\geqslant g'(0)u-Mu^{1+\delta}~\text{for}~u\in[0,p],~~~~h(v)\geqslant h'(0)v-Mv^{1+\delta}~\text{for}~v\in[0,p].
\end{equation}
Let $(w_1(t,x), w_2(t,x))$ denote the solution of the following system
\[
\left\{\begin{aligned}
&\partial_t w_1(t,x)=k_1*w_1(t,x)-w_1(t,x)-\alpha w_1(t,x)+h(w_2(t,x)), ~~t>0,~x\in\mathbb R,\\
&\partial_t w_2(t,x)=k_2*w_2(t,x)-w_2(t,x)-\beta w_2(t,x)+g(w_1(t,x)), ~~~t>0,~x\in\mathbb R,\\
&w_1(0,x)=w_0(x),~w_2(0,x)=w_0(x),~~x\in\mathbb R.
\end{aligned}
\right.
\]
Then Lemma \ref{th4.1} implies that
\begin{equation}\label{05.7}
(u(t,x),v(t,x))\geqslant (w_1(t,x), w_2(t,x))~~\text{for all}~t\geqslant0,x\in\mathbb R.
\end{equation}
Since $k_1(\cdot)$ and $k_2(\cdot)$ are symmetric   and decreasing on $\mathbb R^+$, it follows from Theorem \ref{th2.6}  that $w_1(t,\cdot)$ and $w_2(t,\cdot)$ are also  symmetric and decreasing on $\mathbb R^+$ at any time $t\geqslant0$.

For any given $\lambda\in(0,\lambda^*)$, we define
\[
\left\{
\begin{aligned}
&\underline u(t,x)=\max\big\{0,~\gamma e^{\lambda(-|x|+c(\lambda)t)}-\gamma L e^{\lambda(1+\delta)(-|x|+c(\lambda)t)}\big\},\\
&\underline v(t,x)=\max\big\{0,~\gamma b(\lambda) e^{\lambda(-|x|+c(\lambda)t)}-\gamma L b(\lambda\big(1+\delta)\big) e^{\lambda(1+\delta)(-|x|+c(\lambda)t)}\big\}
\end{aligned}
\right.
\]
for all $t\geqslant0$ and $x\in\mathbb R$,
where $b(\lambda)$ is defined by \eqref{04.5}, $\gamma$ is some positive constant satisfying
\[
0<\gamma\leqslant\min\Big\{\gamma_0,~\frac{\gamma_0}{b(\lambda)}\Big\},
\]
and $L\in\mathbb R^+$ is large enough such that
\begin{equation}\label{05.8}
\begin{aligned}
&L\geqslant\max\bigg\{1,\frac{b(\lambda)}{b(\lambda(1+\delta))},~\gamma^\delta p^{-\delta},~\gamma^\delta p^{-\delta}\frac{ [b(\lambda)]^{1+\delta}}{b\big(\lambda(1+\delta)\big)},\\
&
~~~\frac{M\gamma^{\delta} [b(\lambda)]^{1+\delta}}{G(c(\lambda),\lambda(1+\delta))-h'(0)b(\lambda(1+\delta))},
~\frac{M\gamma^{\delta}}{b(\lambda(1+\delta))H(c(\lambda),\lambda(1+\delta))-g'(0)}\bigg\}.
\end{aligned}
\end{equation}

We easily get that
\[
\underline u(0,x)\leqslant \gamma_0 e^{-\lambda|x|},~~\underline v(0,x)\leqslant \gamma_0 e^{-\lambda|x|}~~\text{for all}~ x\in\mathbb R.
\]
If we consider the function $f(y)=Ay-By^{1+\delta}$ for $y\in\mathbb R^+$ with $A,B\in\mathbb R^+$, whose maximum value equals $f^{\text{max}}\triangleq A^{\frac{1+\delta}{\delta}}B^{-\frac{1}{\delta}}\delta(1+\delta)^{-\frac{1+\delta}{\delta}}$, then we have
\[
\begin{aligned}
&\underline u(t,x)\leqslant f_1^{\text{max}}\triangleq\gamma L^{-\frac{1}{\delta}}\delta(1+\delta)^{-\frac{1+\delta}{\delta}}\leqslant p\leqslant p_1,\\
&\underline v(t,x)\leqslant f_2^{\text{max}}\triangleq\gamma L^{-\frac{1}{\delta}} \big[b(\lambda)\big]^{\frac{1+\delta}{\delta}}\big[b(\lambda(1+\delta))\big]^{-\frac{1}{\delta}}\delta(1+\delta)^{-\frac{1+\delta}{\delta}}\leqslant p\leqslant p_1
\end{aligned}
\]
for all $t\geqslant0$ and $x\in\mathbb R$.
Therefore, the definition of $w_0(\cdot)$ shows that
\begin{equation}\label{05.9}
w_0(x)\geqslant\underline u(0,x),~~w_0(x)\geqslant\underline v(0,x)~~\text{for all}~x\in\mathbb R.
\end{equation}

We now verify that $\big(\underline u(t,x),~\underline v(t,x)\big)$ is a lower solution of system \eqref{1.1}. When $|x|\leqslant c(\lambda)t+(\lambda\delta)^{-1}\ln L$, we easily get  $\underline u(t,x)=0$. Then from (H1) and (H2), it follows that
\[
\partial_t \underline u-k_1*\underline u+\underline u+\alpha \underline u-h(\underline v)\leqslant-h(\underline v)\leqslant0.
\]
When $|x|> c(\lambda)t+(\lambda\delta)^{-1}\ln L$,  we have
\[
\begin{aligned}
&\underline u(t,x)=\gamma e^{\lambda(-|x|+c(\lambda)t)}-\gamma L e^{\lambda(1+\delta)(-|x|+c(\lambda)t)},\\
&\underline v(t,x)\geqslant\gamma b(\lambda) e^{\lambda(-|x|+c(\lambda)t)}-\gamma L b(\lambda\big(1+\delta)\big) e^{\lambda(1+\delta)(-|x|+c(\lambda)t)}.
\end{aligned}
\]
Then by \eqref{05.6}, some simple calculations imply that
\[
\begin{aligned}
&\partial_t \underline u-k_1*\underline u+\underline u+\alpha \underline u-h(\underline v)\\
& \quad \leqslant \gamma\big[G(c(\lambda),\lambda)-h'(0)b(\lambda)\big] e^{\lambda(-|x|+c(\lambda)t)}\\
&\quad\quad \big\{\gamma L\big[G(c(\lambda),\lambda(1+\delta))-h'(0)b\big(\lambda(1+\delta)\big)\big]-M\big[\gamma b(\lambda)\big]^{1+\delta}\big\}e^{\lambda(1+\delta)(-|x|+c(\lambda)t)}.
\end{aligned}
\]
From \eqref{04.6}, \eqref{05.1} and \eqref{05.8}, it follows that
\[
\partial_t \underline u-k_1*\underline u+\underline u+\alpha \underline u-h(\underline v)\leqslant0~~~\text{for}~~|x|> c(\lambda)t+(\lambda\delta)^{-1}\ln L.
\]
Therefore, we finally prove that
\begin{equation}\label{05.10}
\partial_t \underline u-k_1*\underline u+\underline u+\alpha \underline u-h(\underline v)\leqslant0~~\text{for all}~~t>0,~x\in\mathbb R.
\end{equation}
Similarly, we can also prove
\begin{equation}\label{05.11}
\partial_t \underline v-k_2*\underline v+\underline v+\beta \underline v-g(\underline u)\leqslant0~~\text{for all}~~t>0,~x\in\mathbb R.
\end{equation}

From \eqref{05.9}-\eqref{05.11}, Lemma \ref{th4.1} shows that
\[
(w_1(t,x),w_2(t,x))\geqslant(\underline u(t,x),\underline v(t,x))~~\text{for}~~t\geqslant0,~x\in\mathbb R.
\]
Then  some simple calculations imply that
\[
\begin{aligned}
&w_1(t,x)\geqslant\underline u(t,x)=f_1^{\text{max}},~~\text{when}~|x|=c(\lambda)t+(\lambda\delta)^{-1}\ln [(1+\delta)L],\\
&w_2(t,x)\geqslant\underline v(t,x)=f_2^{\text{max}},~~\text{when}~|x|=c(\lambda)t+(\lambda\delta)^{-1}\ln \Big[(1+\delta)L\frac{b(\lambda(1+\delta))}{b(\lambda)}\Big].
\end{aligned}
\]
Since $w_1(t,\cdot)$ and $w_2(t,\cdot)$ are symmetric and decreasing in $\mathbb R^+$ at any time $t\geqslant0$, by taking $\nu=\min\{f_1^{\max},f_2^{\max}\}>0$, we can get from $L\geqslant\max\big\{1,\frac{b(\lambda)}{b(\lambda(1+\delta))}\big\}$ that
\[
w_1(t,x)\geqslant \nu,~~w_2(t,x)\geqslant \nu~~\text{for}~t\geqslant0,~|x|\leqslant c(\lambda)t.
\]
Therefore, by \eqref{05.7} we  prove  that
$
\big(u(t, x),v(t,x)\big)\geqslant(\nu,\nu)~~\text{for all}~~t\geqslant 0,~|x|\leqslant c(\lambda)t.
$
\end{proof}

\begin{remark} \rm
In Theorem \ref{th2.5}, we assume that the initial data $u_0$ and $v_0$ have the same exponentially decaying behavior. When their decaying behaviors are different,  the spatial propagation problem is  more difficult and there are some interesting phenomena. For example, our paper \cite{XLR2017}  shows that the component with exponentially unbounded initial data (for example,   decaying algebraically) can accelerate the component with exponentially decaying  initial data.
However, to the best of our knowledge, when all components decay  exponentially  but their decaying rates are different, there is no study about the interaction among components.
We think that the component with smaller decaying rate could  accelerate that with bigger decaying rate.
The fundamental reason of this acceleration phenomenon is that the growth sources of one component   could come from  other components. For more results about the  acceleration among components, see e.g. Coulon and Yangari \cite{CY2017} and Xu et al. \cite{XLL2018}.
\end{remark}

\section{Applications}

In this section we give some applications of the theoretical results to the control of epidemic whose infectious agent is carried by migratory birds. We consider the question whether it is possible that the epidemic spreads only along the flight route of migratory birds and the spatial propagation against the flight route fails.
Throughout this section, we suppose that the positive parameters $\alpha$, $\beta$, $g'(0)$ and $h'(0)$ in system \eqref{1.1}  have already been determined. Now we assume some specific forms of the kernel functions $k_1$ and $k_2$.

\subsection{Normal distribution}
\noindent

Suppose that the migratory birds fly at a constant speed  $a\in\mathbb R$ and  the infectious agent has its own moving ability. In system \eqref{1.1}, we assume that  $k_1$ and $k_2$ satisfy
\[
k_1(x)=\frac{1}{\sqrt{2\pi\sigma_1}}\exp\left(-\frac{(x-a)^2}{2\sigma_1}\right)~~\text{and}~~k_2(x)=\frac{1}{\sqrt{2\pi\sigma}}\exp\left(-\frac{x^2}{2\sigma}\right).
\]
Here the  expectation $a$ of $k_1$ represents the movements of
infectious agent caused by  migratory flight and  the variance $\sigma_1\in \mathbb R^+$ describes the strength of its own moving  ability.
The  expectation of $k_2$ is $0$  because humans usually return after leaving their own residences.
The variance $\sigma\in \mathbb R^+$   describes  the intensity  of the movements of infectious humans.

By observing the  migration flight of  birds and the moving  ability  of infectious agent, we suppose that the  parameters $a$ and $\sigma_1$ can be determined. We  also  suppose that $a\geqslant0$; otherwise just consider the new spatial variable $y=-x$. Finally, our question becomes  how to restrict the movements of infectious humans such that  the epidemic spreads only along the flight route  and the spatial propagation against the flight route fails; namely we need to find a proper parameter $\sigma$ such that $0<c_l^*<c_r^*$.

Define a constant $r$ which can describe the asymmetry level of $k_1$ as follows
\[
r\triangleq a/\sqrt{2\sigma_1}.
\]

\begin{remark}\label{re6.0}\rm
Intuitively, the asymmetry level of a probability density function $k$ could be measured by the ratio  of $M_1(k)=\int_{\mathbb R^+}k(x)xdx$ to $M_2(k)=\int_{\mathbb R^-}k(x)|x|dx$.
By some calculations, we have that
\[
M_1(k_1)/ M_2(k_1)=\varphi(r) \triangleq2\left(\frac{\exp(-r^2)}{r\sqrt{\pi}}+\text{erf}(r)-1\right)^{-1}+1,
\]
where   $\text{erf}(\cdot)$ is the error function defined by $\text{erf}(r)=\frac{2}{\sqrt{\pi}}\int_0^r \exp{(-t^2)}dt$.
It is easy to check that $\varphi(\cdot)$ is strictly increasing.
Therefore, we can use $r$ to describe the asymmetry level of $k_1$.
\end{remark}
We define another important constant  of system \eqref{1.1} by
\[
\mathcal K \triangleq\beta\left(\alpha+1-\exp(-r^2)\right)\left/\left(g'(0)h'(0)\right)\right. \in\mathbb R^+.
\]
Note that $\mathcal K$ is strictly increasing with respect to $r$.
Next we   show that $\mathcal K$ can describe  the change   of  spatial propagation of system \eqref{1.1}  caused by the asymmetry of $k_1$.

\begin{corollary}\label{pro6.1}
If $\mathcal K >1$, then there is a constant $\sigma^*\in\mathbb R^+$ such that
\begin{itemize}
\item[(i)] when $0<\sigma<\sigma^*$, the spatial propagation against the flight route fails; namely $0<c_l^*<c_r^*$,
\item[(ii)] when $\sigma>\sigma^*$,  the spatial propagation  happens along two directions (along and against the flight route); namely $c_l^*<0<c_r^*$,
\item[(iii)] when $\sigma=\sigma^*$, it is the critical state; namely $0=c_l^*<c_r^*$.
\end{itemize}
Moreover, if $\mathcal K\leqslant1$, then $c_l^*<0<c_r^*$ holds for any $\sigma\in(0,+\infty)$.
\end{corollary}

\begin{proof}
From \eqref{02.6},  some calculations show that
\[
\begin{aligned}
&A(\lambda)=\int_\mathbb R k_1(x)e^{\lambda x}dx-1-\alpha=\exp\left(a\lambda+\frac{\sigma_1}{2}\lambda^2\right)-1-\alpha,\\
&B(\lambda)=\int_\mathbb R k_2(x)e^{\lambda x}dx-1-\beta=\exp\left(\frac{\sigma}{2}\lambda^2\right)-1-\beta.
\end{aligned}
\]
Recall the following sets defined  in the proof of Theorem \ref{th2.2}
\[
\begin{aligned}
&\Lambda^A=\{\lambda\in\mathbb R~\big|~A(\lambda)<0\},~~\Lambda^B=\{\lambda\in\mathbb R~\big|~B(\lambda)<0\},\\
&\Lambda=\big\{\lambda\in\mathbb R~\big|~A(\lambda)B(\lambda)\geqslant g'(0)h'(0),~A(\lambda)<0,~B(\lambda)<0\big\}.
\end{aligned}
\]
We know that $\Lambda^A$ and $\Lambda^B$ are two open intervals and  $\Lambda$ is a closed interval in $\mathbb R$. Moreover, it is easy to check that $\Lambda\subseteq \Lambda^A\cap\Lambda^B$.
Since
\[
\left.\frac{\partial}{\partial \lambda}(B(\lambda)A(\lambda))\right|_{\lambda=0}=-a\beta\leqslant0,
\]
we get that $\Lambda\subseteq\mathbb R^-$ when $a>0$ and $\Lambda=\varnothing$ by $A(0)B(0)<h'(0)g'(0)$ when $a=0$.

Next, in order to  study the relation between $\Lambda$ and $\sigma$, we consider a function  $\Lambda(\cdot): \sigma\mapsto\Lambda$ which is from $\mathbb R^+$ to the set that consists of all closed intervals in $\mathbb R$.
From
\[
\begin{aligned}
&\frac{\partial B}{\partial \sigma}=\frac{1}{2}\lambda^2\exp\left(\frac{\sigma}{2}\lambda^2\right)>0~\text{for}~\lambda\in\mathbb R,\\
&\frac{\partial|AB|}{\partial \sigma}=A\frac{\partial B}{\partial \sigma}<0~~\text{for}~~\lambda\in\Lambda^A\cap\Lambda^B,
\end{aligned}
\]
it follows that
\begin{equation}\label{06.1}
\Lambda(\sigma')\subseteq \Lambda(\sigma)~~\text{for any}~\sigma'>\sigma
\end{equation}
and this inclusion is strict when $\Lambda(\sigma)\neq\varnothing$.
By the continuity of $B$ with respect to $\sigma$,  we know  that $\Lambda(\cdot)$ is also continuous,
which means that  both its lower bound and   upper bound are continuous with respect to $\sigma$ when $\Lambda\neq\varnothing$.

When $\mathcal K>1$,  first, we consider $\sigma\rightarrow 0^+$ and   $\lambda=-a/\sigma_1$, then
\[
\lim\limits_{\sigma\rightarrow 0^+}A\left(-a/\sigma_1\right)B\left(-a/\sigma_1\right)=\beta\left(1+\alpha-\exp\left(-\frac{a^2}{2\sigma_1}\right)\right)>g'(0)h'(0).
\]
Therefore, there is a positive constant $\sigma_0$ small enough such that $\mathrm{int} \Lambda(\sigma_0)\cap \mathbb R^-\neq \varnothing$.
Next, we consider $\sigma\rightarrow +\infty$, then $\lambda_B^+\rightarrow0^+$ and $\lambda_B^-\rightarrow0^-$ where
\[
\lambda_B^{\pm}=\pm\sqrt{\frac{2}{\sigma}\ln(1+\beta)}~~\text{and}~\Lambda^B=(\lambda_B^-,\lambda_B^+).
\]
It follows  that
\begin{equation}\label{06.2}
\lim\limits_{\sigma\rightarrow +\infty}A(\lambda)B(\lambda)\leqslant\alpha\beta<g'(0)h'(0)~~\text{for any}~\lambda\in \Lambda^A\cap\Lambda^B.
\end{equation}
Therefore, there is a positive constant $\sigma_{\infty}$ large enough such that $\Lambda(\sigma_\infty)\cap \mathbb R=\varnothing$.
Finally, by Theorem \ref{th2.2} and \eqref{06.1}, we finish the proof of (i)-(iii) in Corollary \ref{pro6.1}.

When $\mathcal K\leqslant1$, we have
\[
A(\lambda)B(\lambda)\leqslant\beta\left(1+\alpha-\exp\left(-\frac{a^2}{2\sigma_1}\right)\right)\leqslant g'(0)h'(0)~\text{for}~\lambda\in \Lambda^A\cap\Lambda^B.
\]
In the above inequality, the first equality  holds  only if $a=0$, which implies that the second equality does not hold. Then
\[
A(\lambda)B(\lambda)<g'(0)h'(0)~\text{for}~\lambda\in \Lambda^A\cap\Lambda^B,
\]
which means $\Lambda\neq\varnothing$. From  Theorem \ref{th2.2}, it follows that $c_l^*<0<c_r^*$.
\end{proof}

Now we give  more details of  the change  of   spatial propagation  caused by  the asymmetry of $k_1$. When $k_1$ is symmetric (namely $r=0$), it follows that $\mathcal K=\alpha\beta/(h'(0)g'(0))<1$ and  the propagation always happens along two  directions.
When the asymmetry  of $k_1$ becomes stronger (namely, $r$ becomes larger),    $\mathcal K$ becomes larger. If $\mathcal K>1$,  the asymmetry  of $k_1$ is strong enough to change the spreading dynamics of system \eqref{1.1}. It is possible that the epidemic spreads only along the flight route of  migratory birds and the spatial propagation against the flight route fails,
as long as the infectious humans are kept from moving frequently such that $\sigma<\sigma^*$.
Moreover, we point  out that if $(1+\alpha)\beta\leqslant g'(0)h'(0)$, then $\mathcal K<1$  always holds for any  $k_1$,
which means the reaction terms play a more important role and the asymmetry of  dispersal cannot change the  spreading  dynamics of system \eqref{1.1}.

Finally, the critical number $\sigma^*$ can be calculated by some numerical methods. For example, suppose that $\alpha=0.2$, $\beta=0.1$, $h'(0)g'(0)=0.22$, $a=0.5$ and $\sigma_1=1$;
then we have that $\mathcal K=1.4432$ and $\sigma^*=2.2098$.

\subsection{Uniform distribution}
\noindent

Suppose that  $k_1$ and $k_2$ are given by
\[
k_1(x)=\left\{
\begin{aligned}
&\frac{1}{a-b},&\text{for}~x\in[b,a],\\
&0,&\text{for}~x\notin[b,a],
\end{aligned}
\right.~~~\text{and}~~~~
k_2(x)=\left\{
\begin{aligned}
&\frac{1}{2\sigma},&\text{for}~x\in[-\sigma,\sigma],\\
&0,&\text{for}~x\notin[-\sigma,\sigma],
\end{aligned}
\right.
\]
where the constants  $a\in\mathbb R^+$ and $b\in\mathbb R^-$ stand for the farthest distances
of   movements of infectious agent    during a unit time period along and against the flight route, respectively. The  average moving speed is $\int k_1(x)xdx=(a+b)/2$.
The  constant $\sigma\in\mathbb R^+$ stands for the farthest distance of movements of  infectious human  during a unit time period.
Similarly to the normal distribution case, it holds that $\int k_2(x)xdx=0$.
Here the uniform distribution means that every distance in the  moving range  has the same probability to happen.

Similarly to the normal distribution case, we suppose  that the  parameters $a$ and $b$ have already been determined by  some experimental data and  $a+b\geqslant0$;
otherwise,  just  consider the new spatial variable $y=-x$.
Now we show how to choose the  parameter $\sigma$ such that $0<c_l^*<c_r^*$.

From \eqref{02.6}, some calculations show that
\[
A(\lambda)=
\left\{
\begin{aligned}
&\frac{e^{a\lambda}-e^{b\lambda}}{(a-b)\lambda}-1-\alpha,&&&\lambda\neq0,\\
&-\alpha,&&&\lambda=0,
\end{aligned}
\right.\\
\]
\[
B(\lambda)=
\left\{
\begin{aligned}
&\frac{e^{\sigma\lambda}-e^{-\sigma\lambda}}{2\sigma\lambda}-1-\beta,&&\lambda\neq0,\\
&-\beta,&&\lambda=0.
\end{aligned}
\right.
\]
When $a+b> 0$, denote
\[
r=-a/b\in(1,+\infty),
\]
which describes the asymmetry level of $k_1$. Indeed, we have that $M_1(k_1)/ M_2(k_1)=r^2$ and it is strictly increasing with respect to $r$, where  $M_1(k_1)$ and $ M_2(k_1)$ are defined in Remark \ref{re6.0}.

Before giving the result in this case, we need to prove the following lemma.
\begin{lemma}\label{lem6.3}
Let $\omega(z)=(z-1)e^z$ with $z\in\mathbb R$. Then for any $r\in(1,+\infty)$,
there is a unique constant $z_r\in\mathbb R$ such that $\omega(z_r)=\omega(-rz_r)$ and $z_r\neq0$.
Moreover, we have that $z_r\in(1-1/r,1)$.
In addition, when $r=1$,  $\omega(z)>\omega(-z)$ for $z\in\mathbb R^+$.
\end{lemma}

\begin{proof}
For $r\in(1,+\infty)$, define a function
\[
\bar \omega(z)=\omega(z)-\omega(-rz)=(z-1)e^z+(rz+1)e^{-rz}~~\text{for}~z\in\mathbb R.
\]
It follows that  $\bar \omega'(z)=ze^z-r^2ze^{-rz}$ for $z\in\mathbb R$.
Denote  $z_1=0$ and $z_2=2(1+r)^{-1}\ln r\in(0,1)$. Then  some calculations imply  that $\bar \omega'(z_1)=\bar \omega'(z_2)=0$ and
\[
\bar \omega'(z)<0,~z\in(z_1,z_2)~\text{and}~\bar \omega'(z)>0,~z\in\mathbb R \backslash[z_1,z_2].
\]
It is easy to check that
\[
\bar \omega(1)=(r+1)e^{-r}>0
\]
and it follows from $r-1/r>2\ln r$ for $r>1$ that
\[
\bar \omega\left(1-1/r\right)=\frac{e^{1-r}}{r}\left(r^2-e^{r-1/r}\right)<0~\text{for}~r>1.
\]
Then we can find a unique constant $z_r\in(1-1/r,1)$ such that $\bar \omega(z_r)=0$; namely $\omega(z_r)=\omega(-rz_r)$.
Moreover, when $r=1$, we have that $z_1=z_2$ and $\bar \omega$ is strictly increasing in $\mathbb R$. Then $\omega(z)>\omega(-z)$ for $z\in\mathbb R^+$ by $\bar \omega(0)=0$.
\end{proof}

Now define  $\omega(z)=(z-1)e^z$ with $z\in\mathbb R$. From Lemma \ref{lem6.3},  let $z_r$ denote the constant satisfying $\omega(z_r)=\omega(-rz_r)$.
In view of $A'(\lambda)=\frac{1}{(a-b)\lambda^2}(\omega(a\lambda)-\omega(b\lambda))$,  from $\omega(z_r)=\omega(-rz_r)$,
it follows that $A'(z_r/b)=0$ and
\[
A\left(z_r/b\right)=\min\{A(z);z\in\mathbb R\}=\frac{e^{z_r}}{1+rz_r}-1-\alpha\leqslant A(0)<0.
\]
Now we can define the constant $\mathcal K$ which describes  the change  of  the spatial propagation of system \eqref{1.1} caused by the asymmetry of $k_1$, as follows
\[
\mathcal K\triangleq\frac{-\beta \min\{A(z);z\in\mathbb R\}}{g'(0)h'(0)}=\frac{-\beta A\left(z_r/b\right)}{g'(0)h'(0)}>0.
\]
When $a+b=0$, by $\min\{A(z);z\in\mathbb R\}=-\alpha$, we can simply denote $\mathcal K=\alpha\beta/(g'(0)h'(0))$.
From  the following result, we see that   $\mathcal K$ is  strictly increasing with respect to $r$.

\begin{proposition}\label{pro6.4}
\rm
$\frac{\partial}{\partial r}\mathcal K>0$ for $r>1$.
\end{proposition}

\begin{proof}
It suffices to prove that $\frac{\partial}{\partial r}A\left(z_r/b\right)<0$ for $r>1$.
Differentiating the equation $\omega(z_r)=\omega(-rz_r)$  with respect to $r$, we have that
\[
\frac{d z_r}{dr}=\frac{rz_r}{e^{(1+r)z_r}-r^2}.
\]
Then
\[
\begin{aligned}
\frac{\partial}{\partial r}\left(\frac{e^{z_r}}{1+rz_r}\right)&=\frac{e^{z_r}(1-r+rz_r)}{(1+rz_r)^2}\cdot\frac{d z_r}{dr}-\frac{e^{z_r}z_r}{(1+rz_r)^2}\\
&=\frac{e^{z_r}z_r}{(1+rz_r)^2(e^{(1+r)z_r}-r^2)}\left(r+r^2z_r-e^{(1+r)z_r}\right)
\end{aligned}
\]
Also from $\omega(z_r)=\omega(-rz_r)$, it holds that $e^{(1+r)z_r}=(1+rz_r)/(1-z_r)$.
Then by $z_r\in(1-1/r,1)$, we have
\[
r+r^2z_r-e^{(1+r)z_r}=\frac{1+rz_r}{1-z_r}(r-rz_r-1)<0.
\]
From the proof of Lemma \ref{lem6.3}, it holds that $z_r>z_2=2(1+r)^{-1}\ln r$; namely $e^{(1+r)z_r}-r^2>0$.
Therefore, $\frac{\partial}{\partial r}A\left(z_r/b\right)<0$, which completes the proof.
\end{proof}

Now we give the result on the change of spatial propagation caused by the asymmetry of $k_1$.

\begin{corollary}\label{pro6.2}
All the results in Corollary \ref{pro6.1} hold for the uniform distribution case.
\end{corollary}

\begin{proof}
Although this proof is  similar to the proof of  Corollary \ref{pro6.1}, we need to check some details.
Let the sets $\Lambda$, $\Lambda_A$ and $\Lambda_B$ and the function $\Lambda(\cdot): \sigma\mapsto\Lambda$  be the same notations in the proof of  Corollary \ref{pro6.1}.
By some calculations and Lemma \ref{lem6.3}, we have
\[
\frac{\partial B}{\partial \sigma}=\frac{\omega(\lambda \sigma)-\omega(-\lambda \sigma)}{2\lambda \sigma^2}>0~\text{for}~\lambda\in\mathbb R.
\]
Then it follows that \eqref{06.1} holds and this inclusion is strict when $\Lambda(\sigma)\neq\varnothing$.

When $\mathcal K>1$, consider $\sigma\rightarrow 0^+$ and  $\lambda=z_r/b$, then
\[
\lim\limits_{\sigma\rightarrow 0^+}A\left(z_r/b\right)B\left(z_r/b\right)=-\beta A\left(z_r/b\right)>g'(0)h'(0).
\]
Considering $\sigma\rightarrow +\infty$,  we have that  $B(\lambda)\rightarrow +\infty$ for any $\lambda\in\mathbb R^+\cup\mathbb R^-$.
Then  $\lambda_B^+\rightarrow0^+$ and $\lambda_B^-\rightarrow0^-$
where $\Lambda^B=(\lambda_B^-,\lambda_B^+)$,
which means  \eqref{06.2} holds. The rest  of this proof can be obtained similarly.
\end{proof}

From Corollary \ref{pro6.2}, we have some similar discussions to those   from  Corollary \ref{pro6.1} in the normal distribution case. In addition,  here the critical number  $\sigma^*$ can also be  calculated by a numerical method.
For example, when $\alpha=\beta=0.2$, $g'(0)h'(0)=0.06$, $a=2$ and $b=-1$, we have $\mathcal K=1.1952$ and $\sigma^*=0.8423$.

\begin{remark}
\rm
For the more general form of $k_1$, when $k_2$ is symmetric,
we think that Corollary \ref{pro6.1} remains true,
as long as we define $\mathcal K\triangleq\beta(\alpha+1-E(k_1))/(h'(0)g'(0))$ and $\sigma\triangleq \text{Var}(k_2)$,
where  $E(k_1)=\inf\{\int_{\mathbb R}k_1(x)e^{\lambda x}dx;~\lambda\in\mathbb R\}$
and $ \text{Var}(k_2)$ is  the variance  of $k_2$.
\end{remark}

We have presented some applications of the theoretical results to the control of epidemics whose infectious agents (bacteria or viruses) are carried by migratory birds. These applications demonstrate  that the frequent movements of  the infectious humans accelerate  the spreading of the epidemics. Moreover,  it is possible that the epidemic spreads only along the flight route of  migratory birds and the spatial propagation against the flight route fails as long as the infectious humans are kept from moving frequently.

\section*{Acknowledgments}
\noindent

We thank the reviewers for their  helpful comments and Dr. Ru Hou (Peking University) for her helpful discussion. Research of W.-B. Xu was partially supported by China Postdoctoral Science Foundation (2019M660047). Research of W.-T. Li was partially supported by NSF of China (11731005, 11671180). Research of S. Ruan was partially supported by National Science Foundation (DMS-1853622).


\begin{thebibliography}{99}
\bibitem{AO2020}A. Alhasanat, C. Ou, On the conjecture for the pushed wavefront to the diffusive Lotka-Volterra competition model, \emph{J. Math. Biol.} \textbf{80} (2020) 1413-1422.

\bibitem{ABL2007} L.J.S. Allen, B.M. Bolker, Y. Lou, A.L. Nevai, Asymptotic profiles of the steady states for an SIS epidemic patch model, \emph{SIAM J. Appl. Math.} \textbf{67} (2007) 1283-1309.


\bibitem{A2010} F. Andreu-Vaillo, J.M. Maz$\acute{o}$n, J.D. Rossi and J. Toledo-Melero, \emph{Nonlocal Diffusion Problems}, Math. Surveys Monogr. Vol. 165, Amer. Math. Soc., Providence, RI, 2010.

\bibitem{BL2020}X. Bao,  W.-T. Li,  Propagation phenomena for partially degenerate nonlocal dispersal models in time and space periodic habitats, \emph{Nonlinear Anal. Real World Appl.} \textbf{51} (2020) 102975, 26pp.

\bibitem{BLSW2018}X. Bao, W.-T. Li, W. Shen, Z.C. Wang, Spreading speeds and linear determinacy of time dependent diffusive cooperative/competitive systems, \emph{J. Differential Equations} \textbf{265} (2018) 3048-3091.


\bibitem{bates} P.W. Bates, On some nonlocal evolution equations arising in materials science, in: H. Brunner, X.Q. Zhao and X. Zou (Eds.), \emph{Nonlinear Dynamics and Evolution Equations}, Fields Inst. Commun., Vol. 48, Amer. Math. Soc., Providence, RI,  2006, pp.13-52.



\bibitem{CM1981} V. Capasso, L. Maddalena, Convergence to equilibrium states for a reaction-diffusion system modelling the spatial spread of a class of bacterial and viral diseases, \emph{J. Math. Biol.} \textbf{13} (1981) 173-184.

\bibitem{CM1982} V. Capasso, L. Maddalena, A nonlinear diffusion system modelling the spread of oro-faecal diseases, in  {\it ``Nonlinear Phenomena in Mathematical Sciences,''}  (V. Lakshmikantham ed.) Academic Press, New York,  1982, pp. 207-217.

\bibitem{CK1988}V. Capasso, K. Kunisch, A reaction-diffusion system arising in modeling man-environment diseases, \emph{Quart. Appl. Math.} \textbf{46} (1988) 431-450.

\bibitem{CW1997} V. Capasso, R.E. Wilson, Analysis of a reaction-diffusion system modeling man-environment man epidemic, \emph{SIAM J. Appl. Math. } \textbf{57} (1997) 327-346.

\bibitem{CLL2017} R. Cui, K.Y. Lam, Y. Lou, Dynamics and asymptotic profiles of steady states of an epidemic model in advective environments, \emph{J. Differential Equations} \textbf{263} (2017) 2343-2373.

\bibitem{CL2016} R. Cui, Y. Lou, A spatial SIS model in advective heterogeneous environments, \emph{J. Differential Equations} \textbf{261} (2016) 3305-3343.

\bibitem{CY2017} A.C. Coulon, M. Yangari, Exponential propagation for fractional reaction-diffusion cooperative systems with fast decaying initial conditions, \emph{J. Dyn. Diff. Equat.} \textbf{29} (2017) 799-815.

\bibitem{CDM2008} J. Coville, J. D\'{a}vila, S. Mart\'{\i}nez, Nonlocal anisotropic dispersal with monostable nonlinearity, \emph{J. Differential Equations} \textbf{244} (2008) 3080-3118.

\bibitem{Fife2003} P. Fife, Some nonclassical trends in parabolic and parabolic--like evolutions, in: \emph{Trends in Nonlinear Analysis}, Springer, Berlin, 2003, pp.153-191.

\bibitem{FKT2015} D. Finkelshtein, Y. Kondratiev, P. Tkachov, Doubly nonlocal Fisher-KPP equation: front propagation,
\emph{Appl. Anal.} (2019). https ://doi.org/10.1080/00036 811.2019.1643011.

\bibitem{HN2012}F. Hamel, G. Nadin, Spreading properties and complex dynamics for monostable reaction-diffusion equations, \emph{Comm. Partial Differential Equations} \textbf{37} (2012) 511-537.
\bibitem{HR2010}F. Hamel, L. Roques, Fast propagation for KPP equations with slowly decaying initial conditions, \emph{J. Differential Equations} \textbf{249} (2010) 1726-1745.

\bibitem{HY2013}C.-H. Hsu, T.-S. Yang, Existence, uniqueness, monotonicity and asymptotic behaviour of travelling waves for epidemic models, \emph{Nonlinearity} \textbf{26} (2013) 121-139; Erratum: \textbf{26} (2013) 2925-2928.

\bibitem{HKLL2015} C. Hu, Y. Kuang, B. Li, H. Liu,  Spreading speeds and traveling wave solutions in cooperative integral-differential systems, \emph{Discrete Contin. Dyn. Syst. B} \textbf{20} (2015) 1663-1684.

\bibitem{KLSH2010} C.-Y. Kao, Y. Lou, W. Shen, Random dispersal vs non-local dispersal, \emph{Discrete Contin. Dyn. Syst.} \textbf{26} (2010) 551-596.

\bibitem{LLW2002} M.A. Lewis, B. Li, H.F. Weinberger, Spreading speed and linear determinacy for two-species competition models, \emph{J. Math. Biol.} \textbf{45} (2002) 219-233.

\bibitem{LWL2005} B. Li, H.F. Weinberger, M.A. Lewis, Spreading speeds as slowest wave speeds for cooperative systems, \emph{Math. Biosci.} \textbf{196} (2005) 82-98.

\bibitem{LSW2010} W.-T. Li, Y.-J. Sun, Z.-C. Wang, Entire solutions in the Fisher-KPP equation with nonlocal dispersal, \emph{Nonlinear Anal. Real World Appl.} \textbf{11} (2010) 2302-2313.

\bibitem{LXZ2017}W.-T. Li, W.-B. Xu, L. Zhang, Traveling waves and entire solutions for an epidemic model with asymmetric dispersal, \emph{Discrete Contin. Dyn. Syst.} \textbf{37} (2017) 2483-2512.

\bibitem{ly2014} W.-T. Li, F.-Y. Yang, Traveling waves for a nonlocal dispersal SIR model with standard incidence, \emph{J. Integral Equ. Appl.}  \textbf{26} (2014) 243-273.

\bibitem{LZ2008}X. Liang, X.-Q. Zhao, Asymptotic speeds of spread and traveling waves for monotone semiflows with applications, \emph{Comm. Pure Appl. Math.} \textbf{60} (2007) 1-40; Erratum: \textbf{61} (2008) 137-138.

\bibitem{LZ2010}X. Liang, X.-Q. Zhao, Spreading speeds and traveling waves for abstract monostable evolution systems, \emph{J. Functional Analysis} \textbf{259} (2010) 857-903.

\bibitem{LW2020}S. Liu, M. Wang, Existence and uniqueness of solution of free boundary problem with partially degenerate diffusion, \emph{Nonlinear Anal. Real World Appl.} \textbf{54} (2020) 103097, 11pp.
\bibitem{Lui1989}R. Lui, Biological growth and spread modeled by systems of recursions. I. Mathematical theory, \emph{Math. Biosci.} \textbf{93} (1989) 269-295.

\bibitem{LPL2005}F. Lutscher, E. Pachepsky, M.A. Lewis, The effect of dispersal patterns on stream populations, \emph{SIAM J. Appl. Math.} \textbf{65} (2005)  1305-1327.

\bibitem{MO2019}M. Ma, C. Ou, Linear and nonlinear speed selection for mono-stable wave propagations, \emph{SIAM J. Math. Anal.} \textbf{51} (2019)  321-345.

\bibitem{MHO2019}M. Ma, Z. Huang, C. Ou, Speed of the traveling wave for the bistable Lotka-Volterra competition model, \emph{Nonlinearity} \textbf{32} (2019) 3143-3162.


\bibitem{MYH2019} Y. Meng, Z. Yu, C.-H. Hsu, Entire solutions for a delayed nonlocal dispersal system with monostable nonlinearities,  \emph{Nonlinearity} \textbf{32} (2019) 1206-1236.

\bibitem{Mur1993}J.D. Murray, \emph{Mathematical Biology, II, Spatial Models and Biomedical Applications},  third edition, Interdisciplinary Applied Mathematics \textbf{18}, Springer-Verlag, New York, 2003.


\bibitem{Shen-2010-JDE} W. Shen, A. Zhang, Spreading speeds for monostable equations with nonlocal dispersal in space periodic habitats, \emph{J. Differential Equations} \textbf{249} (2010) 747--795.

\bibitem{SZLW2019}Y.-J. Sun, L. Zhang, W.-T. Li, Z.-C. Wang, Entire solutions in nonlocal monostable equations: asymmetric case, \emph{Commun. Pure Appl. Anal.} \textbf{18} (2019) 1049-1072.
\bibitem{Wang2002}X. Wang, Metastability and stability of patterns in a convolution model for phase transitions, \emph{J. Differential Equations} \textbf{183} (2002) 434-461.
\bibitem{Wang2011}H. Wang, Spreading speeds and traveling waves for non-cooperative reaction-diffusion systems, \emph{J. Nonlinear Sci.} \textbf{21} (2011) 747-783.
\bibitem{WC2012}H. Wang, C. Castillo-Chavez, Spreading speeds and traveling waves for non-cooperative integro-difference systems, \emph{Discrete Contin. Dyn. Syst. Ser. B} \textbf{17} (2012) 2243-2266.
\bibitem{WHO2020}H. Wang, Z. Huang, C. Ou, Speed selection for the wavefronts of the lattice Lotka-Volterra competition system, \emph{J. Differential Equations} \textbf{268} (2020) 3880-3902.
\bibitem{WLS2018}J.-B. Wang, W.-T. Li, J.-W. Sun, Global dynamics and spreading speeds for a partially degenerate system with non-local dispersal in periodic habitats, \emph{Proc. Royal Soc. Edinburgh Sect. A} \textbf{148} (2018) 849-880.
\bibitem{Wei1982}H.F. Weinberger, Long-time behavior of a class of biological models, \emph{SIAM J. Math. Anal.} \textbf{13} (1982) 353-396.

\bibitem{WLL2002}H.F. Weinberger, M. A. Lewis, B. Li, Analysis of linear determinacy for spread in cooperative models, \emph{J. Math. Biol.} \textbf{45} (2002) 183-218.

\bibitem{WH2016}S.-L. Wu, C.-H. Hsu, Existence of entire solutions for delayed monostable epidemic models, \emph{Trans. Amer. Math. Soc.} \textbf{368} (2016) 6033-6062.

\bibitem{XLL2018}W.-B. Xu, W.-T. Li, G. Lin, Nonlocal dispersal cooperative systems: acceleration propagation among species, \emph{J. Differential Equations} \textbf{268} (2020) 1081-1105.

\bibitem{XLR2018}W.-B. Xu, W.-T. Li, S. Ruan, The spatial propagation of  nonlocal dispersal equations: the influences of asymmetric kernel, preprint, 2018.

\bibitem{XLR2017} W.B. Xu, W.T. Li, S. Ruan, Fast propagation  for reaction-diffusion cooperative systems, \emph{J. Differential Equations} \textbf{265} (2018) 645--670.

\bibitem{XZ2005}D. Xu, X.-Q. Zhao, Bistable waves in an epidemic model, \emph{J. Dynam. Differential Equations} \textbf{17} (2005) 219-247.

\bibitem{YL2017}F.-Y. Yang, W.-T. Li, Dynamics of a nonlocal dispersal SIS epidemic model, \emph{Commun. Pure Appl. Anal.} \textbf{16} (2017) 781-797.



\bibitem{YLR2019}F.-Y. Yang, W.-T. Li, S. Ruan, Dynamics of a nonlocal dispersal SIS epidemic model with Neumann boundary conditions, \emph{J. Differential Equations} \textbf{267} (2019) 2011-2051.

\bibitem{YZ2015} T. Yi, X. Zou, Asymptotic behavior, spreading speeds, and traveling waves of nonmonotone dynamical systems, \emph{SIAM J. Math. Anal.} \textbf{47} (2015) 3005-3034.


\bibitem{ZLW2017} L. Zhang, W.-T. Li, Z.-C. Wang, Entire solutions in an ignition nonlocal dispersal equation: asymmetric kernel, \emph{Sci. China Math.} \textbf{60} (2017) 1791--1804.

\bibitem{ZLW2016}L. Zhang, W.-T. Li, S.-L. Wu, Multi-type entire solutions in a nonlocal dispersal epidemic model, \emph{J. Dynam. Differential Equations} \textbf{28} (2016) 189-224.

\bibitem{ZLWS2019}L. Zhang, W.-T. Li, Z.-C. Wang, Y.-J. Sun, Entire solutions for nonlocal dispersal equations with bistable nonlinearity: asymmetric case, \emph{Acta Math. Sin.} \textbf{35} (2019)  1771-1794.

\bibitem{ZR2018}G. Zhao, S. Ruan, Spational and temporal dynamics of a nonlocal viral infection model, \emph{SIAM J. Appl. Math.} \textbf{78} (2018) 1594-1940.

\bibitem{ZW2004} X.-Q. Zhao, W. Wang, Fisher waves in an epidemic model, \emph{Discrete Contin. Dyn. Syst. B} \textbf{4} (2004) 1117-1128.

\end{thebibliography}
\end{document}